\documentclass[11pt,letterpaper]{amsart}
\usepackage{amsmath,latexsym,amsfonts,amssymb,amsthm}
\usepackage{hyperref}
\usepackage{mathrsfs}
\usepackage{graphicx,color}
\usepackage[alphabetic]{amsrefs}
\usepackage[all,cmtip]{xy}
\usepackage{enumitem}

\newtheorem{prop}{Proposition}[section]
\newtheorem{thm}[prop]{Theorem}
\newtheorem{cor}[prop]{Corollary}
\newtheorem{lem}[prop]{Lemma}

\theoremstyle{definition}

\newtheorem{defn}[prop]{Definition}

\newtheorem{rem}[prop]{\it Remark}

\newtheorem{thmdefn}[prop]{Theorem-Definition}

\newtheorem*{conj}{Conjecture}
\newtheorem{mthm}{Theorem}

\newtheorem*{claim*}{Claim}

\newcommand{\bP}{\mathbb{P}}
\newcommand{\bC}{\mathbb{C}}
\newcommand{\bR}{\mathbb{R}}

\newcommand{\bQ}{\mathbb{Q}}
\newcommand{\bZ}{\mathbb{Z}}
\newcommand{\bN}{\mathbb{N}}

\newcommand{\tX}{\widetilde{X}}

\newcommand{\tG}{\widetilde{G}}

\newcommand{\cY}{\mathcal{Y}}

\newcommand{\cO}{\mathcal{O}}

\newcommand{\cI}{\mathcal{I}}

\newcommand{\cF}{\mathcal{F}}
\newcommand{\cJ}{\mathcal{J}}

\newcommand{\cD}{\mathcal{D}}

\newcommand{\fa}{\mathfrak{a}}

\newcommand{\fm}{\mathfrak{m}}

\newcommand{\va}{\vec{a}}

\newcommand{\bcF}{\bar{\mathcal{F}}}

\newcommand{\Supp}{\mathrm{Supp}}
\newcommand{\supp}{\mathrm{supp}}

\newcommand{\mult}{\mathrm{mult}}
\newcommand{\lct}{\mathrm{lct}}

\newcommand{\Pic}{\mathrm{Pic}}

\newcommand{\vol}{\mathrm{vol}}
\newcommand{\ord}{\mathrm{ord}}

\newcommand{\Gr}{\mathrm{Gr}}

\newcommand{\pr}{\mathrm{pr}}

\newcommand{\rd}{\mathrm{d}}
\newcommand{\Bs}{\mathrm{Bs}}

\newcommand{\Val}{\mathrm{Val}}

\newcommand{\Y}{Y_{\bullet}}
\newcommand{\W}{W_{\vec{\bullet}}}
\newcommand{\V}{V_{\vec{\bullet}}}
\newcommand{\M}{M_{\vec{\bullet}}}

\numberwithin{equation}{section}

\title{Seshadri constants and K-stability of Fano manifolds}
\author{Hamid Abban}
\author{Ziquan Zhuang}
\date{}

\begin{document}

\address{\emph{Hamid Abban}
\newline
\textnormal{Department of Mathematical Sciences, Loughborough University, Loughborough LE11 3TU, UK
\newline
 \texttt{h.abban@lboro.ac.uk}}}

\address{\emph{Ziquan Zhuang}
\newline 
\textnormal{Department of Mathematics, MIT, Cambridge, MA, 02139, USA
\newline
\texttt{ziquan@mit.edu}}}

\maketitle

\begin{abstract} We give a lower bound of the $\delta$-invariants of ample line bundles in terms of Seshadri constants. As applications, we prove the uniform K-stability of infinitely many families of Fano hypersurfaces of arbitrarily large index, as well as the uniform K-stability of most families of smooth Fano threefolds of Picard number one.
\end{abstract}

\section{Introduction}
Existence of K\"{a}hler-Einstein metrics on Fano manifolds is detected by K-stability: 
a Fano manifold admits a K\"{a}hler-Einstein metric if and only if it is K-polystable \cites{CDS, Tian-YTD}. However, deciding whether a given Fano manifold is K-polystable is a quite challenging problem. We refer to the recent survey by Xu \cite{Xu-survey} for details on the subject and its development. 
A uniform approach to checking K-stability was proposed recently by the authors in \cite{AZ-K-adjunction}, which offers an inductive approach to K-stability and the skeleton of its proof relies on lifting the calculation of the so-called $\delta$-invariants (see Section \ref{sec:K-defn}) to certain flags of subvarieties by adjunction. One particularly useful tool that \cite{AZ-K-adjunction} provided is a K-stability criterion that only involves 
the existence of a linear system $|L|$ 
satisfying a simple numerical condition, so that through each point there is a curve given by complete intersection of divisors in $|L|$ (see \cite{AZ-K-adjunction}*{Theorem\,1.2}). In this article, we provide an even stronger criterion using the Seshadri constant, an invariant that was originally introduced by Demailly \cite{Dem-Seshadri} to measure the local positivity of line bundles.

\begin{mthm}[Theorem \ref{thm:delta and Seshadri}]\label{main-1}
Let $X$ be a projective variety of dimension $n\ge 2$ and let $L$ be an ample line bundle on $X$. Let $x\in X$ and let $S=H_1\cap\cdots\cap H_{n-2}\subseteq X$ be a complete intersection surface passing through $x$, where each $H_i\in |L|$. Assume that $S$ is integral, and smooth at $x$. Then
\[
\delta_x(L)\ge \frac{n+1}{(L^n)}\cdot \varepsilon_x(L|_S).
\]
For a precise description of the equality cases we refer to the statement of Theorem \ref{thm:delta and Seshadri}.
\end{mthm}

This result enables us to give strong estimate of the $\delta$-invariants of many Fano varieties. One major application is to prove uniform K-stability for a large class of smooth hypersurfaces. Before stating the result, we recall the following folklore conjecture; see \cite{Xu-survey}*{Part\,3}.

\begin{conj} 
Any smooth Fano hypersurface $X\subseteq\bP^{n+1}$ of degree $d\ge 3$ is K-stable.
\end{conj}

For general Fano hypersurfaces, this conjecture follows from the K-stability of the Fermat hypersurfaces \cites{Tian-Fermat,AGP,Z-equivariant}, together with the openness of the K-stable locus in smooth families \cites{Oda-openness,Don-openness,Xu-quasimonomial,BLX-openness}. For arbitrary hypersurfaces, the conjecture is known to be true when the Fano index is at most two \cites{Fuj-alpha,LX-cubic-3fold,AZ-K-adjunction}, or when the dimension is at most $4$ \cite{Liu-cubic-4-fold}. Using Theorem \ref{main-1} and some careful study of Seshadri constants, we are able to extend this to a much larger class of hypersurfaces:

\begin{mthm}[Theorem \ref{thm:hypersurface}]\label{main-2}
Let $X\subseteq \bP^{n+1}$ be a smooth Fano hypersurface of Fano index $r\ge 3$ and dimension $n\ge r^3$. Then $X$ is uniformly K-stable.
\end{mthm}

For smooth Fano manifolds, uniform K-stability is equivalent to K-stability as a combination of the analytic works \cites{CDS,Tian-YTD,BBJ-variational}\footnote{Postscript note: an algebraic proof is now available by the very recent work \cite{LXZ-HRFG}}. Our proof of the above theorem is completely algebraic. In particular, even for general hypersurfaces of Fano index $\ge 3$ in the given dimensions, this is perhaps the first algebraic proof of their uniform K-stability (see \cite{AZ-K-adjunction} for the small index cases).

The next application concerns smooth Fano threefolds of Picard number one. They are classified by Iskovskikh into seventeen families; see \cite{Fano-book}. Some of them, such as $\bP^3$, the quadric $Q$, and the Fano threefold $V_5$ of index two and degree $5$, have infinite automorphism groups and therefore are not K-stable. It is also well-known that not all Fano threefolds of degree $22$ are K-stable \cite{Tian-K-stability-defn}. We show that in most of the remaining degrees, the Fano threefolds are uniformly K-stable. This is new when the Fano threefold has index one and degree at least $8$. It also provides a unified and purely algebraic proof for all the sporadic cases that were previously known \cites{AGP,Der-finite-cover,Fuj-alpha,LX-cubic-3fold,Z-cpi,AZ-K-adjunction}.

\begin{mthm}[Theorem \ref{thm:Fano-3-fold}] \label{main-3}
Let $X$ be a smooth Fano threefold of Picard number one. Assume that 
$(-K_X)^3\neq 18,22$ and $X\neq \bP^3,Q$ or $V_5$. Then $X$ is uniformly K-stable.
\end{mthm}

\subsection{Structure of the paper} We set the notation and gather some preliminary results in Section~\ref{sec:pre}. The main technical result of this paper, Theorem~\ref{main-1}, which relates Seshadri constants to stability thresholds, is contained in Section~\ref{sec:sesh-delta}. In Section~\ref{sec:hyp}, we apply this to obtain the first application, Theorem~\ref{main-2}, which concerns the uniform K-stability of hypersurfaces. Finally, in Section~\ref{sec:3folds} we present the other application, Theorem~\ref{main-3}, which concerns the uniform K-stability of Fano threefolds.

\subsection*{Acknowledgements}  We are grateful to Fabian Gundlach, Yuji Odaka, Artie Prendergast-Smith and Chenyang Xu for helpful discussion.  We would also like to thank the referees for careful reading and helpful comments.
HA is supported by EPSRC grants EP/T015896/1 and EP/V048619/1. ZZ is partially supported by NSF Grant DMS-2055531.

\section{Preliminary}\label{sec:pre}

\subsection{Notation and conventions}

We work over $\bC$. Unless otherwise specified, all varieties are assumed to be normal and projective. A pair $(X,\Delta)$ consists of a variety $X$ and an effective $\bQ$-divisor $\Delta$ such that $K_X+\Delta$ is $\bQ$-Cartier. The notions of klt and lc singularities are defined as in \cite{Kol-mmp}*{Definition 2.8}. If $\pi:Y\to X$ is a projective birational morphism and $E$ is a prime divisor on $Y$, then we say $E$ is a divisor over $X$. A valuation on $X$ will mean a valuation $v\colon \bC(X)^*\to \bR$ that is trivial on $\bC^*$. We write $C_X(E)$ (resp.\ $C_X(v)$) for the center of a divisor (resp.\ valuation) and $A_{X,\Delta}(E)$ (resp.\ $A_{X,\Delta}(v)$) for the log discrepancy of the divisor $E$ (resp.\ the valuation $v$) with respect to the pair $(X,\Delta)$ (see \cites{JM-valuation,BdFFU-valuation}). We write $\Val_{X,\Delta}^*$ for the set of nontrivial valuations $v$ on $X$ such that $(X,\Delta)$ is klt at the center of $v$ and $A_{X,\Delta}(v)<\infty$. For any valuation $v$ and any linear series $V$ we denote by $\cF_v$ the filtration on $V$ given by $\cF_v^\lambda V = \{s\in V\,|\,v(s)\ge \lambda\}$. Let $(X,\Delta)$ be a klt pair, let $Z\subseteq X$ be a closed subset and let $D$ be an effective divisor (or an ideal sheaf) on $X$, we denote by $\lct_Z(X,\Delta;D)$ the largest number $\lambda\ge 0$ such that the non-lc locus of $(X,\Delta+\lambda D)$ does not contain $Z$.

\subsection{K-stability and stability thresholds} \label{sec:K-defn}

In this section, we recall the definition of K-stability through stability thresholds.

\begin{defn}[\cites{FO-delta,BJ-delta}] \label{defn:invariants}
Let $(X,\Delta)$ be a projective pair, let $Z\subseteq X$ be a subvariety, and let $L$ be an ample line bundle on $X$. Let $m>0$ be an integer such that $H^0(X,mL)\neq 0$.
\begin{enumerate}[wide]
    \item An $m$-basis type $\bQ$-divisor of $L$ is a $\bQ$-divisor of the form
    \[
    D=\frac{1}{mN_m}\sum_{i=1}^{N_m} \{s_i=0\}
    \]
    where $N_m=h^0(X,mL)$ and $s_1,\cdots,s_{N_m}$ is a basis of $H^0(X,mL)$. We define $\delta_m(L)$ (resp.\ $\delta_{Z,m}(L)$) to be the largest number $\lambda\ge 0$ such that $(X,\Delta+\lambda D)$ is lc (resp.\ lc at the generic point of $Z$) for every $m$-basis type $\bQ$-divisor $D$ of $(X,\Delta)$.
    \item Let $v$ be a nontrivial valuation on $X$. We define
    \[
        T_m(L;v) =\frac{\max\{v(D)\,|\,D\in |mL|\}}{m}
    \]
    and set $T(L;v)=\lim_{m\to\infty} T_m(L;v)$ (it is usually called the pseudo-effective threshold). We say that $v$ is of linear growth if $T(L;v)<\infty$ (this is the case when $v$ is divisorial or $v\in\Val^*_{X,\Delta}$; see \cite{BSKM-linear-growth}*{Section 2.3} and \cite{BJ-delta}*{Section 3.1}). For such valuations we set
    \[
        \vol(L;v\ge t)=\lim_{m\to \infty}\frac{\dim \{s\in H^0(X,mL)\,|\,v(s)\ge mt\}}{m^{\dim X}/(\dim X)!}
    \]
    and $S(L;v)=\frac{1}{\vol(L)}\int_0^\infty \vol(L;v\ge t) \rd t$. If $E$ is a divisor over $X$, we define $S(L;E):=S(L;\ord_E)$, $T(L;E):=T(L;\ord_E)$, etc.
    \item Assume that $(X,\Delta)$ is klt (or klt at the generic point of $Z$ in the local case). The local and global stability thresholds (or $\delta$-invariant) are defined to be
    \[
    \delta_Z(L):=\inf_{v\in\Val^*_{X,\Delta}, Z\subseteq C_X(v)} \frac{A_{X,\Delta}(v)}{S(L;v)}, \quad \delta(L):=\inf_{v\in\Val^*_{X,\Delta}} \frac{A_{X,\Delta}(v)}{S(L;v)}.
    \]
    Clearly $\delta(L)=\inf_{x\in X}\delta_x(L)$. We say that a valuation $v\in\Val^*_{X,\Delta}$ computes $\delta_Z(L)$ (or $\delta(L)$) if it achieves the above infimum. Such valuations always exists by \cite{BJ-delta}*{Theorem E}. We will sometimes write $\delta_Z(X,\Delta;L)$ and $\delta(X,\Delta;L)$ if the pair $(X,\Delta)$ is not clear from the context. By \cite{BJ-delta}*{Theorem A}, we have $\lim_{m\to \infty} \delta_m(X,\Delta)=\delta(X,\Delta)$. By the same proof in \emph{loc.\ cit.\ }we also have $\lim_{m\to \infty} \delta_{Z,m}(X,\Delta)=\delta_Z(X,\Delta)$.
\end{enumerate}
\end{defn}

\begin{thmdefn}[\cites{FO-delta,BJ-delta,Fuj-criterion,Li-criterion}]
Let $X$ be a Fano variety. Then it is K-semistable (resp.\ uniformly K-stable) if and only if $\delta(-K_X)\ge 1$ (resp.\ $\delta(-K_X)>1$).
\end{thmdefn}

\subsection{Seshadri constants, movable thresholds and pseudo-effective thresholds}

Let $X$ be a variety and let $L$ be an ample line bundle on $X$. Let $v$ be a valuation of linear growth on $X$ whose center $C_X(v)$ has codimension at least two. Then the movable threshold $\eta(L;v)$ is defined as
\[\begin{split}
\eta(L;v)=\sup\{&\eta>0\,|\,\text{ for some } m\in\bN,\\
&\,\Bs|\cF_v^{m\eta} H^0(X,mL)| \text{ has codimension at least two}\}.
\end{split}\]
Note that $\eta(L;v)>0$. Indeed, if we choose some sufficiently large integer $m>0$ such that $\cO_X(mL)\otimes \cI_{C_X(v)}$ is globally generated, then as the base locus of this linear system has codimension at least $2$ we get $\eta(L;v)\ge \frac{1}{m}v(\cI_{C_X(v)})>0$.

Let $x\in X$ be a smooth point, let $\pi\colon \tX\to X$ be the blowup of $x$ and let $E$ be the exceptional divisor. Then the Seshadri constant $L$ at $x$ is defined to be (see \cite{Dem-Seshadri} or \cite{Laz-positivity-1}*{Chapter 5})
\[
\varepsilon_x(L):= \sup\{t\ge 0 \,|\,\pi^*L-tE \mbox{ is nef}\} = \inf_{C\subseteq X} \frac{(L\cdot C)}{\mult_x C},
\] 
where the infimum is taken over all irreducible curves $C\subseteq X$ passing through $x$. 
We also denote the pseudo-effective threshold $T(L;E)$ (resp.\ movable threshold $\eta(L;E)$) in this case by $\tau_x(L)$ (resp.\ $\eta_x(L)$). Note that $\varepsilon_x(L)=\eta_x(L)$ when $X$ is a surface. By definition it is easy to see that $\varepsilon_x(L)\le \eta_x(L)\le \tau_x(L)$ and
\[
\tau_x(L)=\sup\{\mult_x D\,|\,0\le D\sim_\bQ L\}.
\]
As we have $(\pi^*L-\varepsilon_x(L)E)^n\ge 0$ (where $n=\dim X$), it follows that $\sqrt[n]{(L^n)}\ge \varepsilon_x(L)$. It is also well-known that $\tau_x(L)\ge \sqrt[n]{(L^n)}$ (see e.g. \cite{Laz-positivity-2}*{Lemma 10.4.12}). When $L$ is very ample, we also have $\eta_x(L)\le \sqrt{(L^n)}$: otherwise we can find two effective $\bQ$-divisors $D_1,D_2\sim_\bQ L$ that have no common components such that $\mult_x D_i>\sqrt{(L^n)}$, and we have 
\[
(L^n)=(D_1\cdot D_2\cdot H_1\cdots\cdots\cdot H_{n-2})\ge \mult_x D_1\cdot \mult_x D_2 >(L^n)
\]
for some general members $H_1,\cdots,H_{n-2}$ of the linear system $|L\otimes \fm_x|$, a contradiction. For later use, we recall some more properties of these invariants.

\begin{lem} \label{lem:unique mult=tau divisor}
Let $L$ be an ample line bundle on a variety $X$ and let $v$ be a valuation of linear growth on $X$ whose center has codimension at least two. Assume that $X$ is $\bQ$-factorial of Picard number one and $\eta(L;v)<T(L;v)$. Then there exists a unique irreducible $\bQ$-divisor $D_0\sim_\bQ L$ on $X$ such that $v(D_0)>\eta(L;v)$. Moreover, we have $v(D_0)=T(L;v)$ and for any effective $\bQ$-divisor $D\sim_\bQ L$ such that $v(D)\ge \eta(L;v)$ we have 
\[
D\ge \frac{v(D) - \eta(L;v)}{T(L;v)-\eta(L;v)}\cdot D_0.
\]
\end{lem}


\begin{proof}
For ease of notation, let $\eta=\eta(L;v)$ and $T=T(L;v)$. We first prove the uniqueness. Suppose that there exist two such irreducible $\bQ$-divisors $D_0$, $D_1$. Let $\lambda=\min\{v(D_0),v(D_1)\}>\eta$. Then for some sufficiently divisible integer $m$ the base locus of $|\cF_v^{m\lambda}H^0(mL)|$ has codimension at least two since $mD_0,mD_1\in |\cF_v^{m\lambda}H^0(mL)|$. Hence $\eta\ge \lambda$, a contradiction. This proves the uniqueness of $D_0$.

For the existence, let $D\sim_\bQ L$ be an effective $\bQ$-divisor on $S$ such that $v(D)>\eta$ (which exists as $T(L;v)>\eta(L;v)$). Since $\rho(X)=1$, we may write $D=\sum \lambda_i D_i$ where $\lambda_i>0$, $\sum \lambda_i=1$ and each $D_i\sim_\bQ L$ is irreducible. As $v(D)>\eta$, at least one of the $D_i$ satisfies $v(D_i)>\eta$. This proves the existence of $D_0$. Since such $D_0$ is unique, it is then clear from the definition of pseudo-effective threshold that $v(D_0)=T$ and moreover we have $v(D_i)\le \eta$ for all $i>0$. Thus $v(D-\lambda_0 D_0)\le (1-\lambda_0)\eta$. Solving this inequality gives $\lambda_0\ge \frac{v(D) - \eta}{T-\eta}$.
\end{proof}

\begin{lem} \label{lem:Seshadri*pseudoeff}
Let $S$ be a $\bQ$-factorial surface of Picard number one and let $x\in S$ be a smooth point. Then we have $\varepsilon_x(L)\cdot \tau_x(L)=(L^2)$.
\end{lem}

\begin{proof}
Again let $\varepsilon=\varepsilon_x(L)$ and $\tau=\tau_x(L)$. Clearly $\varepsilon\le \tau$ by definition. If $\varepsilon=\tau$, then since $\tau\ge \sqrt{(L^2)}\ge \varepsilon$ we necessarily have $\varepsilon=\tau=\sqrt{(L^2)}$, hence $\varepsilon\tau=(L^2)$. Thus we may assume that $\varepsilon<\tau$. By Lemma \ref{lem:unique mult=tau divisor}, there exists a unique irreducible $\bQ$-divisor $D_0\sim_\bQ L$ on $S$ such that $\mult_x D_0=\tau$. Let $C\subseteq S$ be an irreducible curve passing through $x$. If $D_0$ is supported on $C$, then 
\[
\frac{(C\cdot L)}{\mult_x C}=\frac{(D_0\cdot L)}{\mult_x D_0}=\frac{(L^2)}{\tau},
\]
otherwise we have
\[
\frac{(C\cdot L)}{\mult_x C}=\frac{(C\cdot D_0)}{\mult_x C}\ge \mult_x D_0 = \tau.
\]
Since $\tau^2 \ge (L^2)$, by the definition of Seshadri constants we see that $\varepsilon=\frac{(L^2)}{\tau}$.
\end{proof}

\subsection{Restricted volumes}

We refer to \cite{ELMNP} for the original definition of the restricted volume $\vol_{X|Z}(L)$ of a divisor $L$ along a subvariety $Z$. For our purpose their most important properties are summarized in the following statement.

\begin{lem} \label{lem:restricted vol}
Let $L$ be an ample line bundle on a projective variety $X$ of dimension $n$, let $\pi\colon Y\to X$ be a birational morphism and let $E\subseteq Y$ be a prime divisor on $Y$. Then
\begin{equation} \label{eq:vol'=restricted vol}
    \frac{\rd}{\rd t} \vol(\pi^*L-tE) = -n\cdot  \vol_{Y|E}(\pi^*L-tE)
\end{equation}
for all $0\le t<T(L;E)$. Moreover, the function $t\mapsto \vol_{Y|E}(\pi^*L-tE)^{\frac{1}{n-1}}$ is concave on $[0,T(L;E))$. 
\end{lem}

\begin{proof}
The equality \eqref{eq:vol'=restricted vol} follows from \cite{BFJ-volume-C^1} or \cite{LM-okounkov-body}*{Corollary C}. The concavity part follows from \cite{ELMNP}*{Theorem A}.
\end{proof}

We may then rewrite the formula of $S$-invariants using restricted volumes as follows.

\begin{lem} \label{lem:formula for S-inv thru restricted vol}
In the situation of Lemma \ref{lem:restricted vol}, we have
\[
S(L;E) =\frac{n}{(L^n)}\int_0^{T(L;E)} x\cdot \vol_{Y|E}(\pi^*L-xE) \rd x.
\]
\end{lem}

\begin{proof}
By definition, $S(L;E)=\frac{1}{(L^n)}\int_0^{T(L;E)} \vol(\pi^*L-tE)\rd t$. Thus the statement follows from \eqref{eq:vol'=restricted vol} and integration by parts.
\end{proof}

\subsection{Filtered linear series and compatible divisors}

In this section, we briefly recall some definitions from \cite{AZ-K-adjunction}*{Sections 2.5-2.6}.

\begin{defn}
Let $L_1,\cdots,L_r$ be line bundles on $X$. An $\bN^r$-graded linear series $W_{\vec{\bullet}}$ on $X$ associated to the $L_i$'s consists of finite dimensional subspaces 
\[
W_{\va} \subseteq H^0(X, \cO_X(a_1 L_1+\cdots+a_r L_r))
\]
for each $\va\in \bN^r$ such that $W_{\vec{0}}=\bC$ and $W_{\va_1}\cdot W_{\va_2} \subseteq W_{\va_1+\va_2}$ for all $\va_1,\va_2\in \bN^r$. The support $\Supp(\W)\subseteq \bR^r$ of $W_{\vec{\bullet}}$ is defined as the closure of the convex cone spanned by all $\va\in\bN^r$ such that $W_{\va}\neq 0$. In this paper we only consider multi-graded linear series that have bounded support and contain ample series, see \cite{LM-okounkov-body}*{Section 4.3} or \cite{AZ-K-adjunction}*{Definition 2.11} for the precise definition. A filtration $\cF$ on $\W$ is given by the data of vector subspaces $\cF^\lambda W_{\va}\subseteq W_{\va}$ for each $\lambda\in\bR$ and $\va\in\bN^r$ such that $\cF^{\lambda_1} W_{\va_1}\cdot \cF^{\lambda_2} W_{\va_2} \subseteq \cF^{\lambda_1+\lambda_2} W_{\va_1+\va_2}$ for all $\lambda_i\in\bR$ and all $\va_i\in\bN^r$. We only consider linearly bounded filtrations, i.e. there are constants $C_1$ and $C_2$ such that $\cF^\lambda W_{\va}=W_{\va}$ for all $\lambda<C_1|\va|$ and $\cF^\lambda W_{\va}=0$ for all $\lambda>C_2|\va|$. Any valuation $v$ of linear growth on $X$ induces a filtration $\cF_v$ on $\W$ such that $\cF^\lambda_v W_{\va}=\{s\in W_{\va}\,|\,v(s)\ge \lambda\}$.
\end{defn}

If we think of $\W$ as an $\bN\times \bN^{r-1}$-graded linear series and view the first $\bN$-factor as the level grading, then we may define $m$-basis type $\bQ$-divisors of $\W$ as a $\bQ$-divisor of the form
\[
D=\frac{1}{mN_m}\sum_{i=1}^{N_m} \{s_i=0\}
\]
where $s_1,\cdots,s_{N_m}$ enumerate some basis of $W_{m,\va}$ for all $\va\in \bN^{r-1}$ (we call it an $m$-basis of $\W$) and $N_m=\sum_{\va} \dim W_{m,\va}$. Following \cite{AZ-K-adjunction}, we say that $D$ is compatible with a filtration $\cF$ on $\W$ if every $\cF^\lambda W_{m,\va}$ is spanned by some of the $s_i$. Let $\Delta(\W)=\{\va\in\bR^{r-1}\,|\,(1,\va)\in \Supp(\W)\}$. For any $\va\in \bQ^{r-1}$ in the interior of $\Delta(\W)$, we set
\[
\vol_{\W}(\va):=\lim_{m\to \infty} \frac{\dim W_{m,m\va}}{m^n/n!}\quad (n=\dim X)
\]
where the limit is taken over integers $m$ such that $m\va\in \bN^{r-1}$. By \cite{LM-okounkov-body}*{Corollary 4.22}, it extends a continuous function on the interior of $\Delta(\W)$, which we still denote by $\vol_{\W}(\cdot)$. For each $\va\in \bN^r$, we let $M_{\va}$ (resp.\ $F_{\va}$) be the movable (resp.\ fixed) part of $W_{\va}$. We define $F(\W):=\lim_{m\to\infty} F_m (\W)\in \mathrm{Div}(X)_\bR$, where
\[
F_m (\W) = \frac{1}{mN_m}\sum_{\va\in\bN^{r-1}} \dim (W_{m,\va})\cdot F_{m,\va}.
\]
Note that the limit exists by \cite{AZ-K-adjunction}*{Lemma-Definition 2.25}. As in \emph{loc.\ cit.}, we also set $c_1(\W):=\lim_{m\to \infty} c_1(D_m)\in {\rm NS}(X)_\bR$ (where $D_m$ is any $m$-basis type $\bQ$-divisor of $\W$) and $c_1(M_{\bullet}):=c_1(\W)-F(\W)$. 
Many of the invariants we define in Section \ref{sec:K-defn} also generalizes to this setting, see \cite{AZ-K-adjunction}*{Section 2.6}. For our purposes we recall the following.
\begin{enumerate}[wide]
 \item The pseudo-effective threshold $T(\W;\cF)$ of a filtration $\cF$ on $\W$ is defined as 
    \[
        T(\W;\cF):=\lim_{m\to \infty} \frac{T_m(\W;\cF)}{m} = \sup_{m\in \bN} \frac{T_m(\W;\cF)}{m}.
    \]
    where 
\[T_m(\W;\cF)=\sup\{\lambda\in\bR\,|\,\cF^\lambda W_{m,\va}\neq 0 \mbox{ for some } \va\}.\]
  \item We set $S(\W;\cF):=\lim_{m\to \infty} S_m(\W;\cF)$ where \[S_m(\W;\cF)= \frac{1}{mN_m} \sum_{\lambda,\va} \lambda\cdot \dim \Gr_\cF^\lambda W_{m,\va}.\]
    \item Given a closed subset $Z\subseteq X$, we define \[\delta_Z(\W,\cF):=\limsup_{m\to \infty}\delta_{Z,m}(\W,\cF),\] where $\delta_{Z,m}(\W,\cF)=\inf_D \lct_Z(X,\Delta;D)$ and the infimum runs over all $m$-basis type $\bQ$-divisors $D$ of $\W$ that are compatible with $\cF$.
\end{enumerate}

Our multi-graded linear series mostly come from the refinement of a complete linear series.

\begin{defn}[\cite{AZ-K-adjunction}*{Example 2.15}] \label{defn:refinement}
Let $L$ be a big line bundle on $X$ and let $\V$ be the complete linear series associated to $L$, i.e.\ $V_m=H^0(X,mL)$. Let $\pi\colon Y\to X$ be a birational morphism and let $F$ be a Cartier prime divisor on $Y$. The refinement of $\V$ by $F$ is the $\bN^2$-graded linear series $\W$ associated to $\pi^*L|_F$ and $-F|_F$ on $F$ given by
\[
W_{m,j}={\rm Im}(H^0(Y,m\pi^*L-jF)\to H^0(F,m\pi^*L|_F-jF|_F)).
\]
Note that any filtration $\cF$ on $\V$ naturally induces a filtration $\bcF$ on $\W$, i.e.\ $\bcF^\lambda W_{m,j}$ is the image of $\cF^\lambda V_m \cap H^0(Y,m\pi^*L-jF)$.
\end{defn}

\begin{lem} \label{lem:compare S after refinement}
In the above notation, we have $S(\V;\cF)=S(\W;\bcF)$.
\end{lem}

\begin{proof}
It suffices to show that $S_m(\V;\cF)=S_m(\W;\bcF)$. By \cite{AZ-K-adjunction}*{Lemma 3.1}, we may find a basis $s_1,\dots,s_{N_m}$ of $V_m$ that is compatible with both $\cF$ and $\cF_F$, the filtration induced by $F$. By construction, they restrict to form an $m$-basis of $\W$ that is compatible with $\bcF$. Let $\lambda_i=\sup\{\lambda\,|\,s_i\in \cF^\lambda V_m\}$. Then by the definition of $S$-invariants it is easy to see that $S(\V;\cF)=\frac{1}{mN_m}\sum_{i=1}^{N_m} \lambda_i=S(\W;\bcF)$.
\end{proof}

For computations we often choose refinements that are almost complete \cite{AZ-K-adjunction}*{Definition 2.27}. 

\begin{defn}
Let $L$ be a big line bundle on $X$ and let $\W$ be an $\bN^r$-graded linear series. We say that $\W$ is \emph{almost complete} (with respect to $L$) if for every $\va\in {\rm int}(\Supp(\W))\cap \bQ^r$ and all sufficiently divisible integers $m$ (depending on $\va$), we have $|M_{m\va}|\subseteq|L_{m,\va}|$ for some $L_{m,\va}\equiv \ell_{m,\va}L$ and some $\ell_{m,\va}\in \bN$ such that 
\[
\frac{\dim W_{m\va}}{h^0(X,\ell_{m,\va}L)}=\frac{\dim M_{m\va}}{h^0(X,\ell_{m,\va}L)}\to 1
\]
as $m\to \infty$.
\end{defn}

In the surface case, all refinements as in Definition \ref{defn:refinement} are almost complete by \cite{AZ-K-adjunction}*{Lemma 4.10}. Another common example is the refinement of the complete linear series associated to an ample line bundle $L$ by some integral member $H\in |L|$, see \cite{AZ-K-adjunction}*{Example 2.28}.

\section{Seshadri constants and stability thresholds}\label{sec:sesh-delta}

In this section, we prove the following statement, giving lower bounds of stability thresholds in terms of Seshadri constants on complete intersection surfaces. This will be a key tool to verify K-stability of Fano varieties in subsequent sections.

\begin{thm} \label{thm:delta and Seshadri}
Let $X$ be a projective variety of dimension $n\ge 2$ and let $L$ be an ample line bundle on $X$. Let $x\in X$ be a smooth point and let $S=H_1\cap\cdots\cap H_{n-2}\subseteq X$ be a complete intersection surface passing through $x$, where each $H_i\in |L|$. Assume that $S$ is integral and is smooth at $x$. Then
\[
\delta_x(L)\ge \frac{n+1}{(L^n)}\cdot \varepsilon_x(L|_S).
\]
When equality holds, we have at least one of the following:
\begin{enumerate}
    \item $\varepsilon_x(L|_S)=\tau_x(L|_S)=\sqrt{(L^n)}=1$, and $\delta_x(L)$ is computed by any $H\in |L\otimes \fm_x|$, or
    \item $\varepsilon_x(L|_S)=\tau_x(L|_S)=\sqrt{(L^n)}>1$, and the center of any valuation $v$ that computes $\delta_x(L)$ has dimension $\dim C_X(v)\ge n-2$, or
    \item $\varepsilon_x(L|_S)\tau_x(L|_S)=(L^n)$, and every valuation that computes $\delta_x(L)$ is divisorial and induced by a prime divisor $G\subseteq X$ containing $x$ such that $S\not\subseteq G$ and $L\equiv\tau_x(L|_S)G$.
\end{enumerate}
\end{thm}

As one might expect, the careful analysis of the equality cases in the above statement will be useful in proving \emph{uniform} K-stability in several cases. 

Note that an upper bound on the log canonical threshold in terms of Seshadri constants was studied in \cite{Odaka-Sano}. However, the relation in \ref{thm:delta and Seshadri} is in the opposite direction, which offers significantly more flexibility in estimating the $\delta$-invariant.

The proof of Theorem \ref{thm:delta and Seshadri} is by induction on the dimension, where the inductive step is based on \cite{AZ-K-adjunction}*{Lemma 4.6}. Apart from that, the heart of the proof is a detailed analysis of the surface case, where we can be even more precise about the equality cases:

\begin{lem} \label{lem:surface delta>=Seshadri}
Let $S$ be a surface and let $L$ be an ample line bundle on $S$. Let $x\in S$ be a smooth point. Then
\[
\delta_x(L)\ge \frac{3}{(L^2)}\cdot \varepsilon_x(L),
\]
and equality holds if and only if $\varepsilon_x(L)=\tau_x(L)=\sqrt{(L^2)}$, or $\varepsilon_x(L)\tau_x(L)=(L^2)$ and there exists a unique irreducible curve $C\subseteq X$ containing $x$ such that $L\equiv\tau_x(L)C$. Moreover, in the latter case, the curve $C$ is the only divisor that computes $\delta_x(L)$.
\end{lem}

The idea to prove the above statement is to consider the refinement $\W$ of the linear series associated to $L$ by the ordinary blowup of $x$ and then compare the stability thresholds of $L$ and $\W$ using tools from \cite{AZ-K-adjunction}. Using Zariski decomposition on surfaces, we will estimate the stability threshold $\delta(\W)$ in terms of restricted volume functions and reduce the inequality in Lemma \ref{lem:surface delta>=Seshadri} to an inequality of the following type.

\begin{lem} \label{lem:calculus center=pt}
Let $0<a\le b$ and let $g(x)$ be a bounded concave function on $[0,b)$ such that $g(x)=x$ for all $x\in [0,a)$. Then
\[
3a\int_0^b (2x-g(x))\cdot g(x) \rd x \le 4\left(\int_0^b g(x)\rd x\right)^2,
\]
and equality holds if and only if $a=b$, or $g(x)=h(x)$ for all $x\in[0,b)$, where
\[
h(x)=
\begin{cases}
    x & \text{if } 0\le x\le a, \\
    \frac{a(b-x)}{b-a} & \text{if } a < x\le b.
\end{cases}
\]
In particular, when equality holds we have $\int_0^b g(x)\rd x =\frac{1}{2}ab$.
\end{lem}

\begin{proof}
It is straightforward to check that equality holds when $a=b$, thus we may assume that $b>a$. Let $f(x)=g(x+a)-h(x+a)$ and let $c=b-a>0$. Then $f(0)=0$, $f(x)$ is a bounded concave function on $[0,c)$ and the inequality in the statement of the lemma is equivalent to
\[\begin{split}
a^2\int_0^c \left(\frac{6x}{c}-4\right)f(x)\rd x &+ a\int_0^c (6x-4c)f(x)\rd x\\
&-3a\int_0^c f(x)^2 \rd x - 4\left(\int_0^c f(x)\rd x\right)^2 \le 0
\end{split}
\]
by an elementary calculation. We claim that 
\[
\int_0^c (3x-2c)f(x)\rd x \le 0,
\]
which clearly implies the previous inequality as well as the equality condition $f(x)\equiv 0$. To prove the claim, consider $F(t)=\int_0^t (3x-2t)f(x)\rd x$ as a function of $t\in [0,c]$. Then $F(0)=0$ and $F'(t)=t\cdot f(t)-2\int_0^t f(x)\rd x\le t\cdot f(t) - 2 \int_0^t \frac{x}{t}f(t)\rd x = 0$ ($\forall t\in (0,c)$) by the concavity of $f(x)$. Thus $F(c)\le 0$ and we are done.
\end{proof}

To further analyze the equality case in Theorem \ref{thm:delta and Seshadri} and Lemma \ref{lem:surface delta>=Seshadri}, we need two more auxiliary results.

\begin{lem} \label{lem:calculus center=div}
Let $a>0$ and let $g(x)$ be a nonnegative bounded concave function on $[0,a)$ such that $g(0)>0$. Let $n>0$ be an integer. Then
\[
g(0)^{n-1}\int_0^a x\cdot g(x)^{n-1} \rd x \le \frac{n}{n+1}\left(\int_0^a g(x)^{n-1}\rd x\right)^2,
\]
with equality if and only if $n=1$ or  $g(x)=(1-\frac{x}{a})g(0)$.
\end{lem}

\begin{proof}
The result is clear when $n=1$, so we may assume that $n\ge 2$. Up to rescaling, we may also assume that $g(0)=1$. For each $b>0$, let $f_b(x)=1-\frac{x}{b}$ $(0\le x\le b)$. Since $g(x)$ is nonnegative and concave, we have $g(x)\ge f_a(x)$ for all $x\in [0,a]$ and thus $\int_0^a g(x)^{n-1}\rd x\ge \int_0^a f_a(x)^{n-1}\rd x$. As $\lim_{b\to \infty} \int_0^b f_b(x)^{n-1}\rd x = \infty$, by interpolation we know that there exists some $b\ge a$ such that
\begin{equation} \label{eq:conification}
    \int_0^a g(x)^{n-1}\rd x = \int_0^b f_b(x)^{n-1}\rd x.
\end{equation}
It is easy to check that $\int_0^b x\cdot f_b(x)^{n-1} \rd x = \frac{n}{n+1}\left(\int_0^b f_b(x)^{n-1}\rd x\right)^2$, hence it suffices to show
\begin{equation} \label{eq:barycenter shift}
    \int_0^a x\cdot g(x)^{n-1} \rd x\le \int_0^b x\cdot f_b(x)^{n-1} \rd x.
\end{equation}
For ease of notation, set $g(x)=0$ when $a<x\le b$ and set $h(x)=f_b(x)^{n-1}-g(x)^{n-1}$. Since $g(x)$ is concave on $[0,a]$ and $f_b(x)$ is linear, there exists some $c\le a$ such that $h(x)\le 0$ for all $x\in [0,c]$ and $h(x)>0$ for all $x\in (c,b)$. Note that $c>0$ by \eqref{eq:conification}. We then have
\[
\int_0^b x h(x) \rd x = \int_0^c x h(x) \rd x + \int_c^b x h(x) \rd x \ge c \int_0^c h(x) \rd x + c \int_c^b h(x) \rd x = 0,
\]
where the last equality follows from \eqref{eq:conification}. This proves \eqref{eq:barycenter shift}. When equality holds, we have $h(x)=0$, thus $b=a$ and $g(x)=1-\frac{x}{a}$.
\end{proof}

\begin{lem} \label{lem:S-inv of div on X}
Let $L$ be an ample line bundle on a variety $X$ of dimension $n$. Let $G\subseteq X$ be a prime divisor on $X$. Then
\[
S(L;G)\le \frac{(L^n)}{(n+1)(L^{n-1}\cdot G)},
\]
with equality if and only if $L\equiv aG$ for some $a>0$.
\end{lem}

\begin{proof}
The result is clear when $n=1$, so we assume that $n\ge 2$. Let $\pi\colon Y\to X$ be a log resolution such that the strict transform $\tG=\pi^{-1}_*G$ of $G$ is smooth. Let $a=T(L;G)$. By Lemmas \ref{lem:restricted vol} and \ref{lem:formula for S-inv thru restricted vol}, we have 
\[
S(L;G) =\frac{n}{(L^n)}\int_0^a x\cdot \vol_{Y|\tG}(\pi^*L-x\tG) \rd x,
\]
\begin{equation} \label{eq:L^n as integral}
    (L^n) = n \int_0^a \vol_{Y|\tG}(\pi^*L-x\tG) \rd x.
\end{equation}
Since $L$ is ample, by \cite{ELMNP}*{Lemma 2.4} we also have $\vol_{Y|\tG}(\pi^*L)=\vol_{X|G}(L)=(L^{n-1}\cdot G)$, hence the inequality follows directly from Lemma \ref{lem:calculus center=div} applied to $g(x)=\vol_{Y|\tG}(\pi^*L-x\tG)^{\frac{1}{n-1}}$, which is concave by \cite{ELMNP}*{Theorem A}. Suppose that equality holds, then by Lemma \ref{lem:calculus center=div} we have  $g(x)=(1-\frac{x}{a})g(0)$, i.e.
\[
\vol_{Y|\tG}(\pi^*L-x\tG) = \left( 1-\frac{x}{a} \right)^{n-1} (L^{n-1}\cdot G)
\]
for all $0\le x<a$. A direct calculation through \eqref{eq:L^n as integral} then yields $(L^n)=a(L^{n-1}\cdot G)$, or $(L^{n-1}\cdot L-aG)=0$. It follows that $L-aG\equiv 0$ as $L-aG$ is pseudo-effective and $L$ is ample. Clearly $S(L;G)=\frac{a}{n+1}$ if $L\equiv aG$. This finishes the proof.
\end{proof}

We are ready to present the proof of Lemma \ref{lem:surface delta>=Seshadri} and Theorem \ref{thm:delta and Seshadri}.

\begin{proof}[Proof of Lemma \ref{lem:surface delta>=Seshadri}]
Let $\pi\colon T\to S$ be the ordinary blowup at $x$ with exceptional divisor $E\cong \bP^1$. Let $\V$ be the complete linear series associated to $L$ and let $\W$ be its refinement by $E$. Note that $\delta_x(L)=\delta_x(\V)$. Let $\lambda = \frac{3}{(L^2)}\cdot \varepsilon_x(L)$, let $\varepsilon=\varepsilon_x(L)$ and let $\tau=\tau_x(L)=T(L;E)$. Since $\W$ is almost complete by \cite{AZ-K-adjunction}*{Lemma 4.10}, applying \cite{AZ-K-adjunction}*{Corollary 3.4} we know that $\delta_x(\V)\ge \lambda$ as long as
\begin{equation} \label{eq:lambda<=A/S}
    \lambda\le \frac{A_S(E)}{S(L;E)}
\end{equation}
and $\delta(E,\lambda\cdot F(\W);c_1(\M))\ge \lambda$ holds, where $\M$ is the movable part of $\W$. By the definition of stability thresholds, the latter inequality is equivalent to saying
\begin{equation} \label{eq:delta>=lambda}
    \lambda\cdot S(c_1(\M);P) + \lambda \cdot \mult_P F(\W) \le 1
\end{equation}
for all closed point $P\in E$. Let us verify that both conditions \eqref{eq:lambda<=A/S}, \eqref{eq:delta>=lambda} holds in our situation. First, we have
\[
F(\W)=\frac{2}{(L^2)}\int_0^\tau \left( \vol_{T|E}(\pi^*L-xE)\cdot N_\sigma (\pi^*L-xE)|_E \right) \rd x
\]
by \cite{AZ-K-adjunction}*{Lemma 4.13}, and
\[
\vol_{T|E}(\pi^*L-xE) = \left(P_\sigma (\pi^*L-xE)\cdot E\right)
\]
by \cite{ELMNP}*{Corollary 2.17 and Example 2.19}, where $P_\sigma (\cdot)$ (resp.\ $N_\sigma (\cdot)$) denotes the nef (resp.\ negative) part in the Zariski decomposition of a $($pseudo-effective$)$ divisor. In particular, letting $g(x)=\vol_{T|E}(\pi^*L-xE)$ ($0\le x<\tau$), we have 
\[\left(N_\sigma (\pi^*L-xE)\cdot E\right)=((\pi^*L-xE)\cdot E)-\left(P_\sigma (\pi^*L-xE)\cdot E\right) = x-g(x).\] By the definition of Seshadri constant, we also have $g(x)=x$ for all $0\le x\le \varepsilon$. Therefore as $g(x)$ is concave by \cite{ELMNP}*{Theorem A}, Lemma \ref{lem:formula for S-inv thru restricted vol} and Lemma \ref{lem:calculus center=pt} yield
\begin{align*}
    S(L;E)+\deg F(\W) & = \frac{2}{(L^2)}\int_0^\tau x\cdot g(x)\ \rd x + \frac{2}{(L^2)}\int_0^\tau (x-g(x))\cdot g(x) \rd x \\
    & \le \frac{2}{(L^2)}\cdot \frac{4\left(\int_0^\tau g(x)\rd x\right)^2}{3\varepsilon} = \frac{2(L^2)}{3\varepsilon} = \frac{2}{\lambda}. \\
\end{align*}
It follows that 
\begin{equation} \label{eq:2/S+F}
    \lambda\le \frac{2}{S(L;E)+\deg F(\W)}\le \frac{2}{S(L;E)} = \frac{A_S(E)}{S(L;E)}
\end{equation}
which verifies \eqref{eq:lambda<=A/S}. Since $E\cong \bP^1$ is a curve, we have $S(c_1(\M);P)=\frac{1}{2}\deg c_1(\M)$ for any closed point $P\in E$. By \cite{AZ-K-adjunction}*{(3.1)}, we also have $\deg \left(c_1(\M)+F(\W)\right)=\deg c_1(\W)=S(L;E)$. Thus we obtain 
\begin{align*}
    \lambda\cdot S(c_1(\M);P) + \lambda \cdot \mult_P F(\W) & \le \lambda \cdot \deg \left(\frac{1}{2} c_1(\M) + F(\W)\right) \\
    & = \frac{\lambda}{2} \cdot (S(L;E)+\deg F(\W)) \le 1
\end{align*}
for any closed point $P\in E$, which verifies \eqref{eq:delta>=lambda}. Hence according to the discussions at the beginning of the proof, \cite{AZ-K-adjunction}*{Corollary 3.4} implies that $\delta_x(L)=\delta_x(\V)\ge \lambda$ as desired.

It remains to prove the equality conditions. It is straightforward to check that $\frac{A_S(E)}{S(L;E)}=\lambda$ (resp.\ $\frac{A_S(C)}{S(L;C)}=\lambda$) when $\varepsilon=\tau=\sqrt{(L^2)}$ (resp.\ $\varepsilon\tau=(L^2)$ and there exists some curve $C\subseteq X$ containing $x$ such that $L\equiv\tau C$), hence $\delta_x(L) = \lambda$ in either case. Conversely, assume that $\delta_x(L)=\lambda$. If $\delta_x(L)$ is computed by $E$, then by \eqref{eq:2/S+F} we have $F(\W)=0$, hence $N_\sigma(\pi^*L-xE)=0$ and $\pi^*L-xE$ is nef for all $0\le x<\tau$. It follows that $\varepsilon=\tau$. Since $\varepsilon\le \sqrt{(L^2)}\le \tau$, we must have $\varepsilon=\tau=\sqrt{(L^2)}$ as desired. On the other hand, if $E$ does not compute $\delta_x(L)$, then $\epsilon<\tau$ and by the equality description in Lemma \ref{lem:calculus center=pt} we have $(L^2)=2\int_0^\tau g(x)\rd x=\varepsilon\tau$. By \cite{BJ-delta}*{Theorem E} and \cite{AZ-K-adjunction}*{Corollary 3.4}, we also see that $\delta_x(L)$ is computed by some valuations $v$ such that $C_S(v)\neq \{x\}$. Thus the center of $v$ is a curve $C\subseteq S$; in particular, $v$ is divisorial and the curve $C$ also computes $\delta_x(L)$, i.e.\ $S(L;C)=\frac{1}{\lambda}$. By Lemma \ref{lem:S-inv of div on X} and the definition of Seshadri constant, we deduce
\[
\varepsilon\le \frac{(L\cdot C)}{\mult_x C}\le (L\cdot C)\le \frac{(L^2)}{3\cdot S(L;C)}=\frac{\lambda\cdot (L^2)}{3}=\varepsilon,
\]
thus equality holds everywhere. In particular, $L\equiv aC$ and $S(L;C)=\frac{a}{3}$ for some $a>0$ by Lemma \ref{lem:S-inv of div on X}. But since $S(L;C)=\frac{1}{\lambda}=\frac{(L^2)}{3\varepsilon}=\frac{\varepsilon\tau}{3\varepsilon}=\frac{\tau}{3}$, we must have $a=\tau$. Since there can be at most one irreducible curve $C\subseteq S$ containing $x$ such that $L\equiv \tau C$ (otherwise it follows from the definition of Seshadri constants that $\epsilon\ge\tau$), this concludes the proof of the equality cases.
\end{proof}

\begin{cor} \label{cor:surface delta Pic=Z}
Let $S$ be a smooth surface of Picard number one and let $L$ be an ample line bundle on $S$. Let $x\in S$ be a closed point. Then $\delta_x(L)\ge \frac{3}{\tau_x(L)}$ and equality holds if and only if $\varepsilon_x(L)=\tau_x(L)$ or $L\sim_\bQ \tau_x(L)C$ for some irreducible curve $C\subseteq S$ passing through $x$.
\end{cor}

\begin{proof}
This is immediate from Lemma \ref{lem:surface delta>=Seshadri} and Lemma \ref{lem:Seshadri*pseudoeff}.
\end{proof}

\begin{rem}
The corollary is false without the Picard number one assumption. Consider for example $S=\bP^1\times \bP^1$ and take $L$ to be the line bundle of bi-degree $(a,b)$ with $0<a\le b$. Then by \cite{Z-product}*{Theorem 1.2} we know that $\delta(L)=\frac{2}{b}$. On the other hand it is not hard to see that $\tau_x(L)=a+b$ for all $x\in S$ and therefore $\delta_x(L)\ge \frac{3}{\tau_x(L)}$ only when $b\le 2a$.
\end{rem}

\begin{proof}[Proof of Theorem  \ref{thm:delta and Seshadri}]
We prove by induction on the dimension $n$. When $n=2$, the only part that is not covered by Lemma \ref{lem:surface delta>=Seshadri} is the assertion that $\delta_x(L)$ is computed by any $H\in|L|$ when $\varepsilon_x(L)=1=(L^2)$ and $\delta_x(L)=3$. However, this follows immediately from the fact that $\frac{A_X(H)}{S(L;H)}=3$ for any $H\in |L|$ (by Lemma \ref{lem:S-inv of div on X}). When $n\ge 3$, we have 
\[
\delta_x(L|_{H_1})\ge \frac{n}{(L^{n-1}\cdot H_1)} \varepsilon_x(L|_S) = \frac{n}{(L^n)} \varepsilon_x(L|_S)
\]
by induction hypothesis. By \cite{AZ-K-adjunction}*{Lemma 4.6}, we then have
\[
\delta_x(L)\ge \min\left\{ n+1, \frac{n+1}{n}\delta_x(L|_{H_1}) \right\} \ge \min\left\{n+1, \frac{n+1}{(L^n)}\cdot \varepsilon_x(L|_S) \right\}.
\]
Since $\varepsilon_x(L|_S)\le \sqrt{(L|_S^2)}\le (L|_S^2)=(L^n)$, we obtain
\[
\delta_x(L)\ge \frac{n+1}{(L^n)}\cdot \varepsilon_x(L|_S). 
\]
Suppose that equality holds. Let $v$ be any valuation that computes $\delta_x(L)$. By \cite{AZ-K-adjunction}*{Lemma 4.6} and the above discussion, we have either $\delta_x(L)=n+1$ and $\varepsilon_x(L|_S)=(L^n)=1$, or $\delta_x(L|_{H_1})=\frac{n}{(L^n)}\cdot \varepsilon_x(L|_S)$. In the former case, we also have $\tau_x(L|_S)=\varepsilon_x(L|_S)=1$ since $\varepsilon_x(L|_S) = \sqrt{(L|_S^2)}$. By the same argument as in the $n=2$ case, we also know that in this case $\delta_x(L)$ is computed by any $H\in |L|$. In the latter case, by \cite{AZ-K-adjunction}*{Lemma 4.6} we also know that $C_X(v)\not\subseteq H_1$ and that for every irreducible component $Z$ of $C_X(v)\cap H_1$ containing $x$, there exists a valuation $v_1$ on $H_1$ with center $Z$ that computes $\delta_x(L|_{H_1})$. By induction hypothesis, either $\varepsilon_x(L|_S)=\tau_x(L|_S)=\sqrt{(L^n)}>1$ and $\dim C_{H_1}(v_1)\ge n-3$, in which case $\dim C_X(v)=\dim C_{H_1}(v)+1\ge n-2$; or $\varepsilon_x(L|_S)\tau_x(L|_S)=(L^n)$ and the center of $v_1$ on $H_1$ is a prime divisor that does not contain $S$. Suppose that we are in the last case. Then $G=C_X(v)$ is also a prime divisor that does not contain $S$. Since $v$ computes $\delta_x(L)$, we have $\frac{1}{S(L;G)}=\delta_x(L)=\frac{n+1}{(L^n)}\cdot \varepsilon_x(L|_S)$. As in the proof of Lemma \ref{lem:surface delta>=Seshadri}, we then obtain
\[\begin{split}
\varepsilon_x(L|_S)\le \varepsilon_x(L|_S)\cdot \mult_x(G|_S)&\le (L|_S\cdot G|_S)= (L^{n-1}\cdot G)\\
&\le \frac{(L^n)}{(n+1)S(L;G)} = \varepsilon_x(L|_S).
\end{split}
\]
Hence equality holds everywhere and $L\equiv \tau_x(L|_S) G$ by Lemma \ref{lem:S-inv of div on X} as in the proof of Lemma \ref{lem:surface delta>=Seshadri}. This completes the proof.
\end{proof}

\section{Hypersurfaces}\label{sec:hyp}

As a first application of Theorem \ref{thm:delta and Seshadri}, in this section we prove the uniform K-stability of the following hypersurfaces.

\begin{thm} \label{thm:hypersurface}
Let $X\subseteq \bP^{n+1}$ be a smooth Fano hypersurface of Fano index $r\ge 3$ and dimension $n\ge r^3$. Then $X$ is uniformly K-stable.
\end{thm}

The main difficulty we need to overcome in order to apply Theorem \ref{thm:delta and Seshadri} in this situation is that the Seshadri constants on complete intersection surfaces are not always large enough. For example, for any effective $\bQ$-divisor $D\sim_\bQ L$ (where $L$ is the hyperplane class) and any general complete intersection surface $S$ passing through some fixed $x$ we have $\varepsilon_x(L_S)\le \frac{(L_S^2)}{\mult_x (D\cap S)}=\frac{(L^n)}{\mult_x D}$. If there exists some $D$ such that $\mult_x D$ is relatively large (more precisely, if $\tau_x(L) > \frac{n+1}{r}$), then we will not be able to derive $\delta_x(X)\ge 1$ directly through Theorem \ref{thm:delta and Seshadri}. Thus we need to analyze these ``bad'' loci. This is done in the next two lemmas. In particular, it turns out that the ``bad'' locus corresponds exactly to points that support divisors of high multiplicities. 

\begin{lem} \label{lem:high mult div when dim>=4}
Let $X\subseteq \bP^N$ be a smooth variety of dimension $n\ge 4$, Picard number one and degree $d$. Let $x\in X$ be a closed point and let $L$ be the hyperplane class. Assume that for some constant $c>\sqrt{d}$ and for any general hyperplane section $Y\subseteq X$ containing $x$ we have $\tau_x(L_Y)\ge c$. Then $\tau_x(L)\ge c$.
\end{lem}

\begin{proof}
For each $t\in \bP H^0(X,\cO_X(1)\otimes \fm_x)$, let $Y_t\subseteq X_t$ be the corresponding hyperplane sections containing $x$. When $t$ is general, $Y_t$ is smooth by Bertini theorem and has Picard number one by Lefschetz theorem. Since $L$ is very ample, we have $\eta_x(L_{Y_t})\le \sqrt{d}$ (see the remark before Lemma \ref{lem:unique mult=tau divisor}). It then follows from Lemma \ref{lem:unique mult=tau divisor} and the assumption that there exists a unique irreducible $\bQ$-divisor $0\le D_t\sim_\bQ L_{Y_t}$ on $Y_t$ such that $\mult_x D_t\ge c$. By a standard Hilbert scheme argument, we may also assume that $mD_t$ is integral for some fixed integer $m>0$. 

We first treat a special case. Suppose that a general $D_t$ is covered by lines passing through $x$. Let $Z\subseteq X$ be the union of all lines passing through $x$. Then $\Supp(D_t)\subseteq Z\cap Y_t$ and hence $Z$ has codimension at most one. Note that $Z\neq X$ since otherwise $X$ is a cone over its hyperplane section, but as $X$ is smooth it must be a linear subspace and it is easy to see that the assumption of the lemma is not satisfied. If $Z_1,\cdots,Z_k\subseteq Z$ are the irreducible components of codimension one in $X$, then as $\dim Z_i\ge 3$, its image under the projection from $x$ has dimension at least $2$, hence $Z_i\cap Y_t$ is irreducible for general $t$ by Bertini theorem. Since $D_t$ is also irreducible and is swept out by lines containing $x$, we deduce that $\Supp(D_t)=Z_i\cap Y_t$ for some $1\le i\le k$. As $X$ has Picard number one, there exists some $\lambda_i>0$ such that $D=\lambda_i Z_i\sim_\bQ L$. By comparing degrees, we then have $D_t=D|_{Y_t}$. Since $Y_t$ is a general hyperplane section, we also have $\mult_x D = \mult_x D_t\ge c$. This proves the lemma in this special case.

In the sequel, we may assume that $D_t$ is not covered by lines containing $x$. In particular, the projection from $x$ defines a generically finite rational map on $D_t$. Since $\dim D_t\ge 2$, we see that $D_t\cap Y_s$ is irreducible for general $s,t\in \bP H^0(X,\cO_X(1)\otimes \fm_x)$ by Bertini theorem. Note that each $D_t$ is also a codimension two cycle on $X$. If there exists some general $s\neq t$ such that $D_s\cap D_t$ has codimension $4$ (here we need $n\ge 4$ to ensure that $D_s\cap D_t$ is nonempty), then we get
\[
d=\deg (D_s\cdot D_t)\ge \mult_x D_s\cdot \mult_x D_t\ge c^2>d,
\]
a contradiction. Thus $D_s\cap D_t$ contains a divisor on both $D_s$ and $D_t$. Clearly this divisor is contained in $Y_s\cap D_t$, which is irreducible for general $s,t$. It follows that 
\[
\Supp(Y_s\cap D_t)\subseteq \Supp(D_s\cap D_t) \subseteq \Supp(D_s).
\]
Now consider a general pencil $\ell\in \bP H^0(X,\cO_X(1)\otimes \fm_x)$ and let $G\subseteq X$ be the divisor swept out by $\Supp(D_t)$ for general $t\in \ell$. In other words, $G$ is the image of the universal divisor $\cD\subseteq \cY$ under the natural evaluation map ${\rm ev}\colon \cY\to X$, where $\cY\to \ell$ is the corresponding family of hyperplane section. Since $D_t$ is irreducible for general $t$, we see that $\cD$ and $G$ are both irreducible. Since $X$ has Picard number one, we have $G\sim_\bQ rL$ for some $r\in\bQ$. Let $D=\frac{1}{r}G$. We claim that $\mult_x D\ge c$. Indeed, for general $t\in \ell$ and $s\in \bP H^0(X,\cO_X(1)\otimes \fm_x)$, we know that $G\cap Y_s$ is irreducible by Bertini theorem as before and $\Supp(Y_s\cap D_t)\subseteq D_s$ by the previous steps. As $t$ varies, the locus $\Supp(Y_s\cap D_t)$ sweeps out a divisor on $Y_s$, which is necessarily contained in both $D_s$ and $G\cap Y_s$. Since $D_s$ and $G\cap Y_s$ are both irreducible, we deduce that they are proportional to each other. By comparing degrees, we see that $D_s=D|_{Y_s}$. As $Y_s$ is a general hyperplane section, this implies $\mult_x D = \mult_x D_s \ge c$ and finishes the proof.
\end{proof}

\begin{lem} \label{lem:high mult div when dim=3}
Let $X\subseteq \bP^N$ be a smooth threefold of Picard number one and degree $d$. Let $L$ be the hyperplane class and let $x\in X$ be a closed point that is contained in at most finitely many lines on $X$. Assume that a very general hyperplane section $Y\subseteq X$ containing $x$ has Picard number one and for some constant $c>\sqrt[3]{d^2}$, we have $\tau_x(L_Y)\ge c$. Then $\tau_x(L)\ge c$.
\end{lem}

\begin{proof}
Let $Y_t\subseteq X_t$ be the hyperplane sections corresponding to \[t\in\bP H^0(X,\cO_X(1)\otimes \fm_x).\] As in the proof of Lemma \ref{lem:high mult div when dim>=4}, when $t$ is very general, there exists a unique irreducible $\bQ$-divisor $0\le D_t\sim_\bQ L|_{Y_t}$ such that $\mult_x D_t\ge c$. Note that $D_t$ is a curve on $X$. Since $(L^3)=d>\left(\frac{d}{c}\right)^3$, there exists a $\bQ$-divisor $0\le D\sim_\bQ L$ on $X$ such that $\mult_x D>\frac{d}{c}$. Since $X$ has Picard number one, we may further assume as before that $D$ is irreducible. Since 
\[
(D\cdot D_t)=d<\mult_x D\cdot \mult_x D_t,
\]
we see that $\supp(D_t)\subseteq \supp(D)$. By assumption, the projection from $x$ defines a generically finite rational map on $D$, hence $D\cap Y_t$ is irreducible by Bertini theorem and $D_t=D|_{Y_t}$ as in the proof of Lemma \ref{lem:high mult div when dim>=4}. As before this implies $\mult_x D = \mult_x D_t\ge c$.
\end{proof}



We now prove Theorem \ref{thm:hypersurface}. The key point is that, while the ``bad'' locus is in general non-empty, it consists of at most a countable number of points. This allows us to use the K-stability criterion \cite{AZ-K-adjunction}*{Lemmas 4.23} from our previous work and then we can just apply Theorem \ref{thm:delta and Seshadri} to conclude.

\begin{proof}[Proof of Theorem \ref{thm:hypersurface}]
By \cite{AZ-K-adjunction}*{Lemmas 4.23 and 4.25}, it suffices to show that $\delta_Z(X)\ge \frac{n+1}{n}$ for any subvariety $Z\subseteq X$ of dimension $\ge 1$. Let $L$ be the hyperplane class and let $d:=(L^n)=n+2-r$. By assumption $n\ge d$ and $d\ge 26$. We first show that $\tau_x(L)\le \sqrt{d}+1$ for a very general point $x\in Z$. Suppose not, then by Lemma \ref{lem:unique mult=tau divisor} there exists a unique irreducible $\bQ$-divisor $0\le D_x\sim_\bQ L$ such that $\mult_x D_x > \sqrt{d}+1$ for all $x\in Z$ (it is easy to see that $\eta_x(L)\le \sqrt{d}$). By a standard Hilbert scheme argument, we can find an irreducible and reduced divisor $G\subseteq X\times U$ (where $U\subseteq Z$ is an affine open subset) and an integer $m>0$ such that $G\in |\pr_1^*\cO_X(mL)|$ and $G|_{X\times \{x\}}=mD_x$ for very general $x\in Z$. In particular, by \cite{EKL-Seshadri}*{Lemma 2.1} we have $\mult_W G = \mult_x (mD_x) > m(\sqrt{d}+1)$, where $W\subseteq X\times U$ is the graph of the embedding $U\hookrightarrow X$. By \cite{EKL-Seshadri}*{Proposition 2.3}, there exists a divisor $G'\in |\pr_1^*\cO_X(mL)|$ such that $G\not\subseteq \Supp(G')$ and $\mult_W G' > m(\sqrt{d}+1)-1\ge m\sqrt{d}$. Let $D'_x=\frac{1}{m}G'|_{X\times \{x\}}\sim_\bQ L$ for general $x\in Z$, then $D_x$ and $D'_x$ have no common components and $\mult_x D'_x>\sqrt{d}$. It follows that
\[
d=\deg (D_x\cdot D'_x)\ge \mult_x D_x\cdot \mult_x D'_x > (\sqrt{d}+1)\sqrt{d}>d,
\]
a contradiction. Hence $\tau_x(L)\le \sqrt{d}+1$ for a very general point $x\in Z$.

Fix some $x\in Z$ with $\tau_x(L)\le \sqrt{d}+1$ and let $Y\subseteq X$ be a general linear subspace section of dimension $3$ that passes through $x$. By Lefschetz theorem, we know that $Y$ has Picard number one. Let $S\in |H^0(Y,\cO_Y(2L)\otimes \fm_x)|$ be a very general member. By \cite{DGF-Noether-Lefschetz}*{Theorem 1.1}, the surface $S$ also has Picard number one. By Lemma \ref{lem:high mult div when dim>=4}, we have $\tau_x(L_Y)\le \sqrt{d}+1\le \sqrt[3]{d^2}$ (as $d\ge 26$) and hence $\tau_x(2L_Y)=2\tau_x(L_Y)\le 2\sqrt[3]{d^2}$. Since $(2L\cdot C)\ge 2$ for any curves $C\subseteq Y$, the threefold $Y$ does not contain any lines under the embedding given by $|2L|$, thus by Lemma \ref{lem:high mult div when dim=3} we also have $\tau_x(2L_S)\le 2\sqrt[3]{d^2}$. By Corollary \ref{cor:surface delta Pic=Z} we get $\delta_x(2L_S)\ge \frac{3}{2\sqrt[3]{d^2}}$, thus by \cite{AZ-K-adjunction}*{Lemma 4.6}, we have $\delta_x(2L_Y)\ge \frac{2}{\sqrt[3]{d^2}}$, or $\delta_x(L_Y)\ge \frac{4}{\sqrt[3]{d^2}}$. A repeated use of \cite{AZ-K-adjunction}*{Theorem 4.6} then yields
\[
\delta_Z(L)\ge \delta_x(L)\ge \frac{n+1}{4}\delta_x(L_Y)\ge  \frac{n+1}{\sqrt[3]{d^2}},
\]
hence $\delta_Z(X)=\frac{\delta_Z(L)}{r}\ge \frac{n+1}{n}$ as long as $n\ge r\sqrt[3]{d^2}$. Since $n\ge d$, this is automatic when $n\ge r^3$. The proof is now complete.
\end{proof}

\section{Fano threefolds}\label{sec:3folds}

As another application of Theorem \ref{thm:delta and Seshadri}, we prove in this section the uniform K-stability of most Fano threefolds of Picard number one. Recall that the degree of a Fano threefold $X$ of Picard number one is defined to be $(H^3)$, where $H$ is the ample generator of $\Pic(X)$. Using the classification of Fano threefolds \cite{Fano-book}, we may restate Theorem \ref{main-3} as follows.

\begin{thm}\label{thm:Fano-3-fold}
Let $X$ be a smooth Fano threefold of Picard number one. Assume that $X$ has index two and degree at most $4$, or it has index one and degree at most $16$. Then $X$ is uniformly K-stable. 
\end{thm}

We remark that while Fano threefolds of Picard number one have been fully classified (most of them are complete intersections in rational homogeneous manifolds), we only need very little information from this classification. As will be seen below, the two key properties of Fano threefolds we need are the following.
\begin{enumerate}
    \item For any closed point $x\in X$ on the Fano threefold, there exists some smooth member $S\in |H|$ passing through $x$.
    \item Most Fano threefolds $X$ of Fano index one are cut out by quadrics and there are only finitely many lines through a fixed point.
\end{enumerate}
Before presenting the details, let us give a brief summary of the proof of Theorem \ref{thm:Fano-3-fold} and indicate where the above two properties are used. We focus on the case when $X$ has Fano index one (the index two cases are easier). Clearly, point (1) allows us to apply Theorem \ref{thm:delta and Seshadri} and immediately obtain
\begin{equation} \label{eq:delta index one}
    \delta_x(X)\ge \frac{4\varepsilon_x(H|_S)}{(H^3)}.
\end{equation}
Note that $S$ is a K3 surface. Seshadri constants on K3 surfaces have been studied by \cite{Knu-K3-Seshadri} and usually they are easier to compute when the surfaces have Picard number one, thus we seek to choose the surface $S$ carefully so that it not only contains $x$, but also has Picard number one. This is a Noether-Lefschetz type problem with a base point constraint and will be studied in Subsection\,\ref{sec:NL Fano 3fold}. Note that this requirement on $S$ can fail on certain Fano threefolds (e.g.\ quartic threefolds with generalized Eckardt points), but we will show that it is satisfied whenever point (2) holds. Fortunately, in the remaining cases already the trivial bound of $\varepsilon_x(H|_S)$ is enough to imply K-stability through \eqref{eq:delta index one}. Once we find a surface $S$ of Picard number one, the argument is relatively straightforward using Seshadri constants calculations and Theorem \ref{thm:delta and Seshadri}, except when $X$ has degree $16$. In this case we only get $\delta(X)\ge 1$ (i.e.\ K-semistability) and need to study the equality condition a bit further. This is done in Subsection\,\ref{sec:equality conditions} by analyzing the inductive step, which is based on inversion of adjunction, in the proof of Theorem \ref{thm:delta and Seshadri}.

\subsection{Equality conditions in adjunction} \label{sec:equality conditions}

As indicated in the summary of proof above, the goal of this subsection is to further analyze the equality conditions in Theorem \ref{thm:delta and Seshadri}. The main technical results are Lemma \ref{lem:equality in adjunction} and Corollary \ref{cor:equality rho=1}, which list a few constraints that need to be satisfied in order to have equality in the adjunction of stability thresholds. They will play an important role when treating Fano threefolds of degree $16$ and complete intersections of two quadrics. 

\begin{lem} \label{lem:S=n/(n+1)T}
Let $X$ be a projective variety of dimension $n$, let $L$ be an ample line bundle on $X$ and let $v$ be a valuation of linear growth on $X$. Then $S(L;v) = \frac{n}{n+1}T(L;v)$ if and only if 
\begin{equation} \label{eq:vol condition for S=n/(n+1)T}
    \frac{\vol(L;v\ge t)}{(L^n)}=1-\left(\frac{t}{T(L;v)} \right)^n
\end{equation}
for all $0\le t\le T(L;v)$. If in addition $v$ is quasi-monomial, then its center is a closed point on $X$.
\end{lem}

\begin{rem}
In general we have $S(L;v) \le \frac{n}{n+1}T(L;v)$ (see for example \cite{AZ-K-adjunction}*{Lemma 4.2}), thus the above statement gives a description of the equality conditions. When $v$ is divisorial, this is essentially proved in \cite{Fuj-alpha}*{Proposition 3.2}.
\end{rem}

\begin{proof}
Assume that $S(L;v) = \frac{n}{n+1}T(L;v)$. Let $T=T(L;v)$. After replacing $L$ by $rL$ for some sufficiently large integer $r$ we may assume that $L$ is very ample. Let $H\in |L|$ be a general member, let $\W$ be the refinement by $H$ of the complete linear series $\V$ associated to $L$, and let $\cF$ be the filtration on $\W$ induced by the filtration $\cF_v$ on $R(X,L)$. Concretely, $\W$ is $\bN^2$-graded and we have 
\begin{equation} \label{eq:filtered pieces of W}
    \cF^\lambda W_{m,j} = \text{Im}(\cF_v^\lambda H^0(X,\cO_X((m-j)L))\to H^0(H,\cO_H((m-j)L)).
\end{equation}
In particular, $W_{m,j}=H^0(H,\cO_H((m-j)L))$ when $m-j\gg 0$. By the definition of the pseudo-effective threshold $T(L;v)$, it follows that 
\[
\Supp(\W^t)\cap (\{1\}\times \bR) = [0,1-t/T]
\]
where $\W^t$ ($0<t<T$) is the multi-graded linear series given by $W^t_{m,j}=\cF^{mt} W_{m,j}$. By Lemma \ref{lem:compare S after refinement}, we also have $S(L;v)=S(\V;\cF_v)=S(\W;\cF)$. Let 
\[
f(t,\gamma) = \vol_{\W^t}(\gamma) \quad (0<t<T, 0<\gamma<1-t/T).
\]
It is clear that $f(t,\gamma)\le \vol_{\W}(\gamma)=(1-\gamma)^{n-1}(L^n)$. Note that $\vol(\W)=\vol(L)=(L^n)$. We then obtain
\begin{align*}
    S(L;v)=S(\W;\cF) & = \frac{1}{(L^n)}\int_0^T \vol(\W^t)\rd t \\ & =\frac{n}{(L^n)}\int_0^T \rd t \int_0^{1-\frac{t}{T}} f(t,\gamma) \rd \gamma \\
    & \le \int_0^T \rd t \int_0^{1-\frac{t}{T}} n(1-\gamma)^{n-1} \rd \gamma = \frac{n}{n+1}T,
\end{align*}
where the second equality follows from \cite{AZ-K-adjunction}*{Corollary 2.22} while the third equality is implied by \cite{AZ-K-adjunction}*{Lemma 2.23}.
Since $S(L;v)=\frac{n}{n+1}T$ by assumption, we see that $f(t,\gamma)=(1-\gamma)^{n-1}(L^n)$ for all $t,\gamma$. By another application of \cite{AZ-K-adjunction}*{Lemma 2.23} (used in the second equality below), we then have 
\begin{align*}
    \vol(L;v\ge t)=\vol(\W^t) & =n\int_0^{1-\frac{t}{T}} f(t,\gamma) \rd \gamma \\
    & = \int_0^{1-\frac{t}{T}} n(1-\gamma)^{n-1}(L^n) \rd \gamma = \left( 1- \left(\frac{t}{T}\right)^n\right)(L^n),
\end{align*}
which proves \eqref{eq:vol condition for S=n/(n+1)T}. Conversely, it is easy to check that $S(L;v)=\frac{n}{n+1}T$ as long as \eqref{eq:vol condition for S=n/(n+1)T} holds. This proves the first part of the lemma.

Suppose next that $v$ is quasi-monomial and $\dim C_X(v)\ge 1$. Let $\pi\colon Y\to X$ be a log resolution and let $Z=C_Y(v)$. Since $\dim C_X(v)\ge 1$, we have $(Z\cdot \pi^*L)\neq 0$ by the projection formula, hence $H_Y$ intersects $Z$ where $H_Y$ is the strict transform of $H$. By Izumi's inequality we have $\lct_Z(f)\ge \frac{1}{\mult_Z(f)}$. On the other hand $\lct_Z(f)\le \frac{A_Y(v)}{v(f)}$ by definition, thus $v(f)\le A_Y(v) \mult_Z(f)$ for any $f\in \cO_{Y,Z}$. Up to rescaling of the valuation $v$ we may assume that $A_Y(v)=1$. Thus by \eqref{eq:filtered pieces of W} we see that $\cF^\lambda W_{m,j}\subseteq H^0(H_Y,\cI^\lambda_{H_Y\cap Z}((m-j)\pi^*L))$. As $\pi^*L$ is nef and big, by \cite{FKL-volume}*{Theorem A(v)} we obtain $\vol_{\W^t}(\gamma)<\vol_{\W}(\gamma)$ for all $0<t,\gamma\ll 1$. By the proof above it then follows that $S(L;v)<\frac{n}{n+1}T$. This proves the second part of the lemma.
\end{proof}

\begin{lem} \label{lem:equality in adjunction}
Let $X$ be a projective variety of dimension $n\ge 2$ with klt singularities and let $L$ be an ample line bundle on $X$ such that the linear system $|L|$ is base point free. Assume that $\delta(L)$ is computed by some valuation $v$ on $X$ with $\dim C_X(v)\ge 1$ and
\[
\delta(L)\le\frac{n+1}{n}\delta_Z(L|_H)
\]
for some general member $H\in|L|$ and some irreducible component $Z$ of $C_X(v)\cap H$. Then one of the following holds:
\begin{enumerate}
    \item $\frac{A_X(v)}{T(L;v)}< \frac{n-1}{n+1}\delta(L)$, or
    \item $\frac{A_X(v)}{T(L;v)}= \frac{n-1}{n+1}\delta(L)$, $\dim C_X(v)=1$, and $\frac{\vol(L;v\ge t)}{(L^n)}=1-n\left(\frac{t}{T}\right)^{n-1}+(n-1)\left(\frac{t}{T}\right)^n$ for all $0\le t\le T:=T(L;v)$.
\end{enumerate}
\end{lem}

\begin{rem}
Note that this implies $\alpha(L)\le \frac{n-1}{n+1}\delta(L)$, which is stronger than the usual inequality $\alpha(L)\le \frac{n}{n+1}\delta(L)$ (see for example \cite{BJ-delta}*{Theorem A}).
\end{rem}

\begin{proof}
Let $\W$ be the refinement by $H$ of the complete linear series associated to $L$ and let $L_0=L|_H$. Note that $\W$ is almost complete, $F(\W)=0$ and $c_1(\W)=\frac{n}{n+1}L_0$ by \cite{AZ-K-adjunction}*{Example 2.28 and (3.1)}. Let $\fa_\bullet(v)$ be the valuation ideals of $v$, i.e., $\fa_m(v)=\{f\in \cO_X\,|\,v(f)\ge m\}$. Let $\fa_m=\fa_m(v)|_H$ and let $v_0$ be a quasi-monomial valuation on $H$ with center $Z$ that computes $\lct_Z(H;\fa_\bullet)$, which exists by \cite{Xu-quasimonomial}*{Theorem 1.1}. After rescaling, we may assume that $A_H(v_0)=A_X(v)$. By inversion of adjunction, $\lct_Z(H;\fa_\bullet)\le \lct_Z(X;\fa_\bullet(v))$. Since $v$ calculates $\lct_Z(X;\fa_\bullet(v))=A_X(v)$ by \cite{BJ-delta}*{Proposition 4.8}, we deduce that $v_0(\fa_\bullet)\ge 1$ and hence $\fa_m(v)|_H\subseteq \fa_m(v_0)$ for all $m$. We now define two filtrations on $\W$: the first one, denoted by $\bcF_v$, is the restriction of the filtration $\cF_v$ on $X$ induced by the valuation $v$, while the second one $\cF_{v_0}$ is induced by the valuation $v_0$. From the previous argument we see that $\cF_{v_0}$ dominates $\bcF_v$, i.e.\ $\bcF_v^\lambda \W \subseteq \cF_{v_0}^\lambda \W$ for all $\lambda\ge 0$. By \cite{AZ-K-adjunction}*{Corollary 2.22}, this implies
\begin{equation} \label{eq:compare two S}
    S(\W;\bcF_v)\le S(\W;\cF_{v_0})=S(\W;v_0),
\end{equation}
and when equality holds we have $T(\W;\bcF_v)=T(\W;\cF_{v_0})$. From the construction, it is clear that $T(\W;\bcF_v)=T(L;v)$ and $T(\W;\cF_{v_0})=T(L_0;v_0)$. By Lemma \ref{lem:compare S after refinement}, we also have
\[
S(L;v)=S(\W;\bcF_v),
\]
thus as $v$ computes $\delta(L)$ and $H\in |L|$ is general, we deduce that
\[
A_H(v_0)=A_X(v)=\delta(L) S(L;v)\le \delta(L) S(\W;v_0).
\]
On the other hand, we have $S(\W;v_0)=\frac{n}{n+1}S(L_0;v_0)$ by \cite{AZ-K-adjunction}*{Lemma 2.29}. Combined with our assumptions we obtain 
\[
\delta(L)S(\W;v_0) \le  \frac{n+1}{n}\delta_Z(L_0)\cdot \frac{n}{n+1}S(L_0;v_0)=\delta_Z(L_0) S(L_0;v_0) \le A_H(v_0).
\]
Therefore equality holds everywhere (including in \eqref{eq:compare two S}) and we get
\[
\delta(L)=\frac{n+1}{n}\delta_Z(L_0), \quad A_X(v)=A_T(v_0)=\delta_Z(L_0)S(L_0;v_0)
\]
i.e.\ $v_0$ computes $\delta_Z(L_0)$, and
\[
T(L;v)=T(\W;\bcF_v)=T(\W;\cF_{v_0})=T(L_0;v_0).
\]
Note that $S(L_0;v_0)\le \frac{n-1}{n}T(L_0;v_0)$ by \cite{AZ-K-adjunction}*{Lemma 4.2}. It follows that 
\[
A_X(v)\le \frac{n-1}{n}\delta_Z(L_0) T(L_0;v_0)=\frac{n-1}{n+1}\delta(L)T(L;v).
\]
If equality holds, then by Lemma \ref{lem:S=n/(n+1)T} we know that $Z$ is a closed point and 
\[
\frac{\vol(L_0;v_0\ge t)}{(L^n)}=1-\left(\frac{t}{T}\right)^{n-1},
\]
where $0\le T\le T=T(L_0;v_0)=T(L;v)$. It follows that $\dim C_X(v)=1$. Let $\W^t$ be the multi-graded linear series given by $W^t_{m,j}=\cF_{v_0}^{mt} W_{m,j}$. Since $W_{m,j}=H^0(H,\cO_H((m-j)L))$ when $m-j\gg 0$, we see that
\[
\vol_{\W^t}(\gamma)=\vol((1-\gamma)L_0;v_0\ge t)
\]
and hence 
\[
\frac{\vol_{\W^t}(\gamma)}{(L^n)}=(1-\gamma)^{n-1}-\left(\frac{t}{T}\right)^{n-1}
\]
for all $0\le \gamma<1-\frac{t}{T}$. Since equality holds in \eqref{eq:compare two S}, using \cite{AZ-K-adjunction}*{Lemma 2.23} we obtain
\begin{align*}
    \vol(L;v\ge t)=\vol(\W^t) & =n(L^n)\int_0^{1-\frac{t}{T}} \left( (1-\gamma)^{n-1}-\left(\frac{t}{T}\right)^{n-1} \right) \rd \gamma \\
    & = \left( 1-n\left(\frac{t}{T}\right)^{n-1}+(n-1)\left(\frac{t}{T}\right)^n \right)(L^n)
\end{align*}
as desired.
\end{proof}

\begin{cor} \label{cor:equality rho=1}
Under the assumption of Lemma \ref{lem:equality in adjunction}, assume in addition that $X$ is $\bQ$-factorial, $\rho(X)=1$ and $C_X(v)$ has codimension $\ge 2$ in $X$. Then at least one of the following holds:
\[
\frac{A_X(v)}{T(L;v)}< \frac{n-1}{n+1}\delta(L), \quad \text{or} \quad \frac{A_X(v)}{\eta(L;v)}\le \frac{n-1}{n+1}\delta(L).
\]
\end{cor}

\begin{proof}
Denote $T:=T(L;v)$ and $\eta:=\eta(L;v)$. It suffices to show that $\eta=T$ in the second case of Lemma \ref{lem:equality in adjunction}. Suppose not, i.e.\ $\eta<T$. Then by Lemma \ref{lem:unique mult=tau divisor}, there exists an irreducible $\bQ$-divisor $D_0\sim_\bQ L$ on $X$ such that $v(D_0)=T$ and for any $t\in (\eta,T)$ and any effective $\bQ$-divisor $D\sim_\bQ L$ with $v(D)\ge t$ we have $D= \frac{t-\eta}{T-\eta}D_0+\frac{T-t}{T-\eta}G$ where $G\sim_\bQ L$ is effective and $v(G)\ge \eta$. It follows that 
\[
\vol(L;v\ge t)=\left(\frac{T-t}{T-\eta}\right)^n \vol(L;v\ge \eta),
\]
which contradicts the expression from Lemma \ref{lem:equality in adjunction} (note that $n\ge 3$ since the curve $C_X(v)$ has codimension $\ge 2$ in $X$). Thus $\eta=T$ and we are done.
\end{proof}

\subsection{The index two case} \label{sec:index two}

As an application of Theorem \ref{thm:delta and Seshadri} and the results from the previous subsection, we now prove: 

\begin{thm} \label{thm:index two}
Let $X$ be a Fano threefold of Picard number one, Fano index two and degree at most $4$. Then $X$ is uniformly K-stable.
\end{thm}

\begin{proof}
Let $H$ be the ample generator of $\Pic(X)$ and let $d=(H^3)$ be the degree of $X$. Using the classification of Fano threefolds (see for example \cite{Fano-book}), it is straightforward to check that for any closed point $x\in X$ there exists some smooth member $S\in |H|$ passing through $x$ (see for example the proof of \cite{AZ-K-adjunction}*{Corollary 4.9(5)} for the degree $1$ case, the other cases are much easier). By adjunction, $S$ is a del Pezzo surface of degree $d$ and $H|_S\sim -K_S$. Since $\delta_x(X)=\frac{1}{2}\delta_x(H)$, it suffices to show $\delta_x(H)>2$.

Suppose first that $d=1$. By \cite{Bro-dP-Seshadri}*{Th\'eor\`eme 1.3}, we know that $\varepsilon_x(-K_S)\ge \frac{1}{2}$, hence by Theorem \ref{thm:delta and Seshadri} we obtain $\delta_x(H)\ge 2$. Moreover the equality cannot hold since $1=(H^3)\neq \left(\frac{1}{2}\right)^2$ and $H$ is a primitive element in $\Pic(X)$. Thus $\delta_x(H)>2$ and we are done in this case.

Similarly, when $d=2$, by \cite{Bro-dP-Seshadri}*{Th\'eor\`eme 1.3} we know that $\varepsilon_x(-K_S)\ge 1$, thus the same argument as above gives $\delta_x(H)>2$.

Suppose next that $d=3$. If $x\in X$ is a generalized Eckardt point, then $\delta_x(X)=\frac{6}{5}>1$ by \cite{AZ-K-adjunction}*{Theorem 4.18}. If $x\in X$ is not a generalized Eckardt point, then there are only finitely many lines on $X$ passing through $x$, thus if $S\subseteq X$ is general then $x$ is not contained in any lines on $S$. By \cite{Bro-dP-Seshadri}*{Th\'eor\`eme 1.3}, we have $\varepsilon_x(-K_S)=\frac{3}{2}$, hence $\delta_x(H)\ge 2$ by Theorem \ref{thm:delta and Seshadri}. Moreover, equality cannot hold as before. Thus $\delta_x(H)>2$ and we are done in this case.

Finally assume that $d=4$. Note that $X$ is a complete intersection of two quadrics in this case. There is a pencil of tangent hyperplanes at any closed point $x$ and every line on $X$ passing through $x$ is contained in the base locus of this pencil, which is a complete intersection curve of degree $4$. It follows that there are at most $4$ lines containing $x$ on $X$. In particular, we may arrange that $x$ is not contained in any lines on $S$. By \cite{Bro-dP-Seshadri}*{Th\'eor\`eme 1.3}, we know that $\varepsilon_x(-K_S)\ge 2$, hence Theorem \ref{thm:delta and Seshadri} yields $\delta_x(H)\ge 2$. Suppose that equality holds. Then by Theorem \ref{thm:delta and Seshadri} there exists some valuation $v$ with positive dimensional center on $X$ such that $\frac{A_X(v)}{S(H;v)}=2$. We show that this is impossible. Indeed, if the center $C_X(v)$ is a prime divisor $D$ on $X$, then we may assume that $v=\ord_D$. Since $\Pic(X)$ is generated by $H$, we have $D\sim rH$ for some $r\ge 1$. By Lemma \ref{lem:S-inv of div on X} we have $S(H;v)=\frac{1}{4r}$, hence $\frac{A_X(v)}{S(H;v)}=4r>2$, a contradiction. Hence we may assume that the center $C_X(v)$ is a curve $C$ on $X$. If $T\in |H|$ is a general member and $x\in T\cap C_X(v)$, then $\varepsilon_x(-K_T)\ge 2$ and $\delta_x(-K_T)\ge \frac{3}{2}$ by \cite{Bro-dP-Seshadri}*{Th\'eor\`eme 1.3} and Theorem \ref{thm:delta and Seshadri} as above. Thus the assumptions of Corollary \ref{cor:equality rho=1} are satisfied and we deduce that either $A_X(v)<T(H;v)$ or $A_X(v)\le \eta(H;v)$. The first case is impossible by \cite{CS-lct-3fold}*{Theorem 6.1}, thus it remains to exclude the other possibility. 

By the definition of the movable threshold $\eta(H;v)$, for any $0<\varepsilon\ll 1$ we can find two effective $\bQ$-divisors $D_1,D_2\sim_\bQ H$ without common components such that $v(D_i)\ge (1-\varepsilon)\eta(H;v)$ ($i=1,2$). Let $m$ be a sufficiently divisible integer and let $Z\subseteq X$ be the complete intersection subscheme $mD_1\cap mD_2$. Note that $\deg Z=4m^2$. If $\deg C\ge 2$, then $\mult_C Z\le 2m^2$ and by \cite{dFEM-mult-and-lct}*{Theorem 0.1}, applied at the generic point of $C$, we have $\lct_C(X;\cI_Z)\ge \frac{\sqrt{2}}{m}$. In particular, 
\[
A_X(v)\ge \frac{\sqrt{2}}{m} v(\cI_Z) = \sqrt{2} \min\{v(D_1),v(D_2)\}\ge \sqrt{2}(1-\varepsilon)\eta(H;v)>\eta(H;v),
\]
a contradiction. If $\deg C=1$, i.e.\ $C$ is a line on $X$, then $\mult_C Z\le 4m^2$ and by \cite{Fano-book}*{Proposition 3.4.1(ii)}, we have $\mult_C D_i<\frac{3}{2}<2$ for some $i=1,2$. Hence $\cI_Z\not\subseteq \cI_C^{\frac{3}{2}m}$ and by \cite{Z-bsr-loc-closed}*{Lemma 2.6}, applied at the generic point of $C$, we see that there exists some absolute constant $\varepsilon_1>0$, which in particular does not depend on $D_i$ or $\varepsilon$, such that $\lct_C(X;\cI_Z)>\frac{1+\varepsilon_1}{m}$. It then follows as before that $A_X(v)>\frac{1+\varepsilon_1}{m} v(\cI_Z)\ge (1+\varepsilon_1)(1-\varepsilon)\eta(H;v)>\eta(H;v)$, a contradiction. Hence the equality $\delta_x(H)=2$ never holds and $X$ is uniformly K-stable. 
\end{proof}

\subsection{Noether-Lefschetz for prime Fano threefolds} \label{sec:NL Fano 3fold}

In this subsection, we prove the following Noether-Lefschetz type result on Fano threefolds. As explained at the beginning of the section, this is another key ingredient in our study of uniform K-stability when the Fano threefolds have index one.

\begin{thm} \label{thm:Noether-Lefschetz}
Let $X$ be a Fano threefold of Picard number one, Fano index one and degree $\ge 6$. Let $x\in X$ be a closed point. Then a very general hyperplane section $S\in |-K_X|$ passing through $x$ has Picard number one.
\end{thm}

\begin{rem}
The assumption that $X$ has degree $\ge 6$ is indeed necessary, since the statement fails on quartic threefolds with generalized Eckardt points.
\end{rem}

For the proof of the theorem we first recall the following criterion.

\begin{lem} \label{lem:NL criterion}
Let $X\subseteq \bP^N$ be a smooth projective threefold, let $\ell\subseteq (\bP^N)^*$ be a Lefschetz pencil of hyperplane sections and let $Y=X\cap H$ where $H$ is a very general member of $\ell$. Assume that the natural map $H^{2,0}(X)\to H^{2,0}(Y)$ is not surjective.
Then the restriction $\Pic(X)\to \Pic(Y)$ is an isomorphism.
\end{lem}

\begin{proof}
This should be well known to experts. By assumption, the vanishing cohomology $H^2(Y,\bC)_{\rm van}$, the orthogonal complement of $H^2(X,\bC)$ relative to the intersection form on $H^2(Y,\bC)$, has non-trivial intersection with $H^{2,0}(Y)$, hence is not generated by algebraic classes. By \cite{SGA7II}*{Th\'eor\`eme 1.4}, this implies that $H^2(Y,\bC)_{\rm van}\cap \Pic(Y)=\{0\}$ and therefore every line bundle on $Y$ is a pullback from $X$; in other words, $\Pic(X)\to \Pic(Y)$ is surjective. It is also injective by the Lefschetz hyperplane theorem.
\end{proof}

In the remaining part of this section, let $X$ be a smooth Fano threefold as in Theorem \ref{thm:Noether-Lefschetz} and let $x\in X$ be a closed point. It is well known (see for example \cite{Fano-book}) that the anti-canonical linear system $|-K_X|$ induces an embedding $X\subseteq \bP^N$. Let $\check{\bP}\subseteq (\bP^N)^*$ be the dual projective space parametrizing hyperplanes of $\bP^N$ containing $x$. To apply Lemma \ref{lem:NL criterion}, we need to find a Lefschetz pencil in $\check{\bP}$. 

\begin{lem} \label{lem:Lefschetz pencil exist}
A general pencil $\ell\subseteq \check{\bP}$ is a Lefschetz pencil.
\end{lem}

\begin{proof}
We first recall some notation and results from \cite{Voisin-Hodge2}*{Section 2.1.1}. Let $Z\subseteq X\times (\bP^N)^*$ be the algebraic subset defined by
\[
Z=\{(y,H)\,|\,X_H:=X\cap H \mbox{ is singular at } y\}.
\]
Let $\cD_X:={\rm pr}_2(Z)\subseteq (\bP^N)^*$ be the set of singular hyperplane sections and let $\cD_X^0\subseteq \cD_X$ be the subset of hyperplanes $H$ such that $X_H$ has at most one ordinary double point as singularity. Let $W:=\cD_X\setminus \cD_X^0$. By \cite{Voisin-Hodge2}*{Corollary 2.8, and the comments thereafter}, we have $\dim W\le N-2$ and $\check{\bP}$ intersects transversally with $\cD_X^0$ away from $Z_x:={\rm pr}_1^{-1}(x)\subseteq Z$. Since $\dim Z_x = N-4 = \dim \check{\bP} -3$, a general line $\ell\subseteq \check{\bP}$ is disjoint from $Z_x$ and thus intersects transversally with $\cD_X^0$. By \cite{Voisin-Hodge2}*{Proposition 2.9}, $\ell$ is a Lefschetz pencil if and only if $\ell$ is also disjoint from $W$. This would be the case if $W\cap \check{\bP}$ (set-theoretic intersection) has codimension at least $2$ in $\check{\bP}$. Clearly $W\cap \check{\bP}=W_1\cup W_2$ where $W_1$ (resp.\ $W_2$) parametrizes hyperplanes $H$ containing $x$ such that $X_H$ has degenerate singularities (resp.\ more than one ordinary double point). It suffices to show that both $W_i$ ($i=1,2$) have codimension at least $2$ in $\check{\bP}$.

To this end, consider the closed subset 
\[
R:=\{y\in X\,|\,x\in \Lambda_y\}\subseteq X,
\]
where $\Lambda_y:=T_y X\subseteq \bP^N$ denotes the tangent space of $X$ at $y$. We claim that
\begin{equation} \label{eq:dim R<=1}
    \dim R\le 1.
\end{equation}
Assuming this claim for now, let us finish the proof of the lemma. Let $\check{Z}=\pr_2^*\check{\bP}$. By construction, $\check{Z}$ is a divisor in $Z$ and $\pr_1 \colon \check{Z}\to X$ is a $\bP^{N-5}$-bundle away from $R$; in particular, $\check{Z}$ is smooth outside $\pr_1^{-1}(R)$. Note that by \eqref{eq:dim R<=1}, we have 
\begin{equation} \label{eq:dim preimage of R}
    \dim \pr_1^{-1}(R)\le N-3<\dim \check{Z}=N-2.
\end{equation}
Let $(y,H)\in \check{Z}\setminus \pr_1^{-1}(R)$. As in \cite{Voisin-Hodge2}*{Lemma 2.7 and Corollary 2.8}, $X_H$ has an ordinary double point at $y$ if and only if the restriction of $\pr_2$ to $\check{Z}$ is an immersion at $(y,H)$, and in that case $\pr_{2*}(T_{\check{Z},(y,H)})$ can be identified with the linear subspace of $\bC^N=T_{(\bP^N)^*,H}$ consisting of functions that vanish at both $x$ and $y$. By generic smoothness and \eqref{eq:dim preimage of R}, this immediately implies that $\dim W_1\le N-3$. On the other hand, let $H$ be a general smooth point of an irreducible component of $W_2$. By construction, $X_H$ has at least two ordinary double points $y_1,y_2$. Since the Fano threefold $X$ is cut out by quadrics and cubics \cite{Isk-Fano-3fold-II}*{Corollary 2.6}, the line $\ell'$ joining $y_1$ and $y_2$ is contained in $X$, otherwise $\ell'$ has intersection number at least $4$ with one of those quadrics or cubics. It follows that either $x\in \ell'$, in which case $y_1,y_2\in R$ and $H\in \pr_2(\pr_1^{-1}(R))$, or $x\not\in\ell'$, in which case $T_{W_2,H}\subseteq \pr_{2*}(T_{\check{Z},(y_1,H)})\cap \pr_{2*}(T_{\check{Z},(y_2,H)})$ has dimension at most $N-3$. In either case we conclude that $\dim W_2\le N-3$. In other words, both $W_1$ and $W_2$ have codimension at least $2$ in $\check{\bP}$. As explained earlier, this implies the statement of the lemma.

It remains to prove claim \eqref{eq:dim R<=1}. Let $y\in R\setminus \{x\}$. If $\deg X\ge 8$, then since $X$ is cut out by quadrics \cite{Isk-Fano-3fold-II}*{Corollary 2.6} and $x\in \Lambda_y$, the line joining $x$ and $y$ is contained in $X$: otherwise as it is tangent to $X$ at $y$, it has intersection number at least $3$ with one of the quadrics, a contradiction. Since there are only finitely many lines on $X$ passing through $x$ \cite{Fano-book}*{Proposition 4.2.2}, we deduce that $\dim R\le 1$. If $\deg X=6$, then $X$ is the complete intersection of a quadric (denoted by $Q$) and a cubic. Again for any $y\in R\setminus \{x\}$, the line joining $x$ and $y$ is contained in $Q$, thus $y\in H_0$, where $H_0$ is the tangent hyperplane of $Q$ at $x$. It follows that $R\subseteq H_0\cap X$. Since $\Pic(X)=\bZ\cdot [H_0]$, we see that $H_0\cap X$ is irreducible and reduced. As there are only finitely many lines passing through $x$, the linear projection from $x$ is finite on $H_0\cap X$, therefore by generic smoothness, it cannot be ramified everywhere. In other words, there exists some smooth point $y_0\in H_0\cap X$ such that $x\not\in T_{y_0}(H_0\cap X)=H_0\cap T_{y_0} X$, or equivalently, $x\not\in T_{y_0} X$. Thus $R\subsetneq H_0\cap X$ and we deduce that $\dim R\le 1$. This finishes the proof.
\end{proof}

We are now ready to prove Theorem \ref{thm:Noether-Lefschetz}.

\begin{proof}[Proof of Theorem \ref{thm:Noether-Lefschetz}]
By Lemma \ref{lem:Lefschetz pencil exist}, there exists a Lefschetz pencil of hyperplane sections passing through $x$. A smooth member $Y$ of the pencil is a smooth K3 surface, thus $\dim H^{2,0}(Y)=1$; on the other hand, as $X$ is Fano, we have $\dim H^{2,0}(X)=0$. In particular, the map $H^{2,0}(X)\to H^{2,0}(Y)$ is not surjective and the theorem follows immediately from Lemma \ref{lem:NL criterion}.
\end{proof}

\subsection{The index one case} \label{sec:index one}

We are now ready to prove Theorem \ref{thm:Fano-3-fold}. By Theorem \ref{thm:index two}, the remaining case is:

\begin{thm} \label{thm:index one}
Every smooth Fano threefold $X$ of Picard number one, Fano index one and degree $d\le 16$ is uniformly K-stable.
\end{thm}

\begin{proof}
Let $H=-K_X$. We will denote by $H_Y$ the restriction of the hyperplane class $H$ to a subvariety $Y\subseteq X$. Let $x\in X$ and let $S\in |H|$ be a very general hyperplane section containing $x$. When $d\le 4$, it is easy to see that $\varepsilon_x(H_S)\ge 1$ since $H$ is base point free, hence $\delta_x(H)>1$ and $X$ is uniformly K-stable by Theorem \ref{thm:delta and Seshadri} as in the proof of Theorem \ref{thm:index two}. Thus for the rest of the proof we may assume that $d\ge 6$. By Theorem \ref{thm:Noether-Lefschetz}, $\Pic(S)=\bZ\cdot [H_S]$. 

We claim that $\tau_x(H_S)\le 4$. Suppose not, then for some integer $m>0$ there exists an integral curve $C\sim -mK_X|_S$ containing $x$ such that $\mult_x C>4m$ (we can assume $C$ is integral since $S$ has Picard number one). Since $\mult_x C$ is an integer, we have $\mult_x C \ge 4m+1$. By adjunction, $K_S\sim 0$. It follows that
\begin{align*}
    16m^2\ge dm^2 & = (K_S+C)\cdot C = 2p_a(C)-2 \\
     & \ge \mult_x C\cdot (\mult_x C-1)-2\ge 4m(4m+1)-2 > 16m^2,
\end{align*}
a contradiction. Hence $\tau_x(H_S)\le 4$ and by Lemma \ref{lem:Seshadri*pseudoeff} and Theorem \ref{thm:delta and Seshadri} we obtain $\delta_x(H)\ge \frac{4}{\tau_x(H_S)}\ge 1$. Moreover, when equality holds, we have $\varepsilon_x(H_S)=\tau_x(H_S)=4$ or $H\sim_\bQ 4G$ for some prime divisor $G$ on $X$. The latter case cannot occur since $H$ is a primitive generator of $\Pic(X)$. By Lemma \ref{lem:Seshadri*pseudoeff}, the former case can only happen when $d=16$. This proves that $X$ is uniformly K-stable when $d\le 14$ and is K-semistable when $d=16$.

It remains to analyze the equality conditions when $X$ has degree $d=16$. Suppose that $\delta_x(X)=1$, then by Theorem \ref{thm:delta and Seshadri} there exists some valuation $v$ with positive dimensional center on $X$ such that $A_X(v)=S(H;v)$. Since $\Pic(X)=\bZ\cdot [H]$, using Lemma \ref{lem:S-inv of div on X} it is easy to see that $S(H;D)\le \frac{1}{4}<1=A_X(D)$ for any prime divisor $D$ on $X$ (c.f. the proof of Theorem \ref{thm:index two}), hence $C_X(v)$ cannot be a surface and must be a curve.



Suppose first that $C$ has degree $(H\cdot C)\ge 2$. We claim that in this case $\delta_C(X)>1$. To see this, let $T\in |H|$ be a very general hyperplane section such that $\Pic(T)=\bZ\cdot[H_T]$ and that $T\cap C$ consists of at least two points, and let $G\in |H_T|$ be a general hyperplane section on $T$ that is disjoint from $T\cap C$. Let $\W$ be the refinement by $T$ of the complete linear series associated to $H$. Note that $\W$ is almost complete, $F(\W)=0$ and $c_1(\W)=\frac{3}{4}H_T$ by \cite{AZ-K-adjunction}*{(3.1)}. Consider the admissible flag $\Y$ on $X$ given by $Y_0=X$, $Y_1=T$ and $Y_2=G$. By \cite{AZ-K-adjunction}*{Theorem 3.5}, we have
\begin{equation} \label{eq:adj along curve deg>=2}
    \delta_C(X;-K_X)\ge \min\left\{\frac{A_X(T)}{S(H;T)},\delta_{T\cap C}(T;\W,\cF)\right\}=\min\{4,\delta_{T\cap C}(T;\W,\cF)\},
\end{equation}
where $\cF$ is the filtration induced by the curve $G$. In particular, $\delta_C(X)>1$ as long as we have $\delta_{T\cap C}(T;\W,\cF)>1$.

We show that this is indeed the case. Let $m\in\bN$ and let $D$ be an $m$-basis type $\bQ$-divisor of $\W$ that is compatible with $\cF$. Then we have
\[
D=S_m(\W;G)\cdot G + \Gamma
\]
for some effective $\bQ$-divisor $\Gamma$. As $G$ is disjoint from $T\cap C$, it is clear that
\begin{equation} \label{eq:lct equal}
    \lct_{T\cap C}(T;D)=\lct_{T\cap C}(T;\Gamma).
\end{equation}
Since
\[
\lim_{m\to\infty} S_m(\W;G)=S(\W;G)=\frac{3}{4}S(H_T;G)=\frac{1}{4}
\]
by \cite{AZ-K-adjunction}*{Lemma 2.29}, we see that $\Gamma\sim_\bQ c_1(\W)-S_m(\W;G)\cdot G\sim_\bQ \lambda_m H_T$ for some $\lambda_m\ge 0$ with $\lim_{m\to\infty} \lambda_m=\frac{1}{2}$. By \cite{AZ-K-adjunction}*{Lemma 2.21} and the last part of its proof, we also know that there exists some $\eta_m\in (0,1)$ with $\lim_{m\to\infty} \eta_m = 1$ such that $\eta_m \cdot S_m(\W;F) < S(\W;F)$ for any divisor $F$ over $T$. In particular, we have 
\[
\ord_F(\eta_m \Gamma)\le \eta_m \cdot S_m(\W;F) < S(\W;F)\le \frac{1}{4}
\]
for any irreducible curve $F\subseteq T$ (recall that $\Pic(T)$ is generated by $H_T$, thus by \cite{AZ-K-adjunction}*{Lemma 2.29}, $S(\W;F)=\frac{3}{4}S(H_T;F)=\frac{1}{4r}$ if $F\sim rH_T$). Perturbing the $\eta_m$, we may also assume that $\eta_m\lambda_m<\frac{1}{2}$. It follows that $(T,4\eta_m\Gamma)$ is klt outside a finite number of points and $2H_T-(K_T+4\eta_m\Gamma)$ is ample (recall that $K_T=0$ by adjunction). 

We now apply an argument from \cite{Z-cpi} to estimate $\lct_{T\cap C}(T,\Gamma)$. More precisely, let $\cJ=\cJ(T,4\eta_m\Gamma)$ be the multiplier ideal, which is co-supported at a finite number of points by the previous step. By Nadel vanishing, $H^1(T,\cJ(2H_T))=0$, hence $\ell(\cO_T/\cJ)\le h^0(T,2H_T)=2(H_T^2)+2=34$. As $|T\cap C|\ge 2$, we see that $\ell(\cO_{T,x}/\cJ_x)\le 17$ for some $x\in T\cap C$. On the other hand, by \cite{Z-cpi}*{Lemmas 3.4 and 5.2}, we have $\lct_x(T,\fa)>\frac{1}{3}$ for any ideal $\fa\subseteq\cO_{T,x}$ with $\ell(\cO_{T,x}/\fa)\le 21\le \bar{\sigma}_{2,3}$ (we follow the notation of \cite{Z-cpi}). Since such ideals $\fa$ can be parametrized by some scheme of finite type and the log canonical thresholds are constructible in families, we deduce that there exists some absolute constant $\alpha>\frac{1}{3}$ such that $\lct_x(T,\fa)>\alpha$ for all $\fa\subseteq\cO_{T,x}$ with $\ell(\cO_{T,x}/\fa)\le 21$. In particular, we have $\lct_x(T,\cJ)\ge\alpha$. By \cite{Z-cpi}*{Theorem 1.6 and Remark 3.1}, we then have $\lct_x(T,4\eta_m\Gamma)\ge \frac{\alpha}{\alpha+1}$ and therefore
\[
\lct_{T\cap C}(T,\Gamma)\ge \lct_x(T,\Gamma)\ge \frac{4\alpha\eta_m}{\alpha+1}.
\]
Combined with \eqref{eq:lct equal} and letting $m\to \infty$, we obtain
\[
\delta_{T\cap C}(T;\W,\cF)\ge \frac{4\alpha}{\alpha+1}>1.
\]
By \eqref{eq:adj along curve deg>=2}, this implies that $\delta_C(X)>1$ for any curve $C\subseteq X$ of degree at least $2$.

Hence the only remaining possibility for $C_X(v)$ is a line $L$ on $X$. The rest of the argument is similar to those of Theorem \ref{thm:index two}.
By Corollary \ref{cor:equality rho=1}, we have $A_X(v)<\frac{1}{2}T(H;v)$ or $A_X(v)\le\frac{1}{2}\eta(H;v)$. In the former case, there exists some $\bQ$-divisor $0\le D\sim_\bQ H$ such that $\lct_L(X;D) < \frac{1}{2}$. Since $X$ has Picard number one, we may further assume that $D$ is irreducible. Note that $\mult_L D>2$, otherwise $(X,\frac{1}{2}D)$ is log canonical \cite{Kol-mmp}*{Theorem 2.29}. By \cite{Fano-book}*{Theorems 4.3.3(vii), Remark 4.3.4 and 4.3.7(iii)}, $2D$ is integral and $\mult_L D=\frac{5}{2}$. Moreover, if $\rho\colon \tX\to X$ is the ordinary blowup of the line $L$ with exceptional divisor $F$ and $D'$ is the strict transform of $D$, then the scheme theoretic intersection $G=2D'\cap F$ contains a non-hyperelliptic curve $\Gamma$, $\rho|_G\colon G \to L$ is finite of degree $5$ and
\[
\rho^*(K_X+\frac{1}{2}D)=K_{\tX}+\frac{1}{4}F+\frac{1}{2}D'.
\]
It follows that $\rho|_\Gamma\colon \Gamma\to L$ has degree at least $3$ and each component of $G\setminus \Gamma$ is different from $\Gamma$ and has multiplicity at most $2$. From here we deduce that every component of $\frac{1}{2}D'\cap F$ has multiplicity $\le \frac{1}{2}$ and by inversion of adjunction we see that $(\tX,\frac{1}{4}F+\frac{1}{2}D')$ is klt over the generic point of $L$ and the same is true for $(X,\frac{1}{2}D)$. But this is a contradiction as $\lct_L(X;D)<\frac{1}{2}$. Therefore we must have $A_X(v)\le\frac{1}{2}\eta(H;v)$. For any $0<\varepsilon\ll 1$ we can find effective $\bQ$-divisors $D_1,D_2\sim_\bQ H$ without common components such that $v(D_i)>(1-\varepsilon)\eta(H;v)$. Let $m$ be a sufficiently divisible integer and let $Z\subseteq X$ be the complete intersection subscheme $mD_1\cap mD_2$. Note that $\mult_L Z\le \deg Z=16m^2$ and by \cite{Fano-book}*{Theorems 4.3.3(vii)}, we have $\mult_C D_i < \frac{5}{2}<4$ for some $i=1,2$. Hence $\cI_Z\not\subseteq \cI_C^{\frac{5}{2}m}$ and by \cite{Z-bsr-loc-closed}*{Lemma 2.6}, applied at the generic point of $C$, we see that there exists some absolute constant $\varepsilon_1>0$, which does not depend on $D_i$ and $\varepsilon$, such that $\lct_C(X;\cI_Z)>\frac{1+\varepsilon_1}{2m}$. It then follows as in the proof of Theorem \ref{thm:index two} that $A_X(v)>\frac{1+\varepsilon_1}{2m} v(\cI_Z) >\frac{1}{2}\eta(H;v)$, a contradiction.

Therefore we conclude that the inequality $\delta_x(H)\ge 1$ is always strict. In other words, $X$ is uniformly K-stable.
\end{proof}

\begin{proof}[Proof of Theorem \ref{thm:Fano-3-fold}]
This is just a combination of Theorems \ref{thm:index two} and \ref{thm:index one}.
\end{proof}


\end{document}